\documentclass[onefignum,onetabnum]{siamart220329}
\newcommand{\rev}[1]{{\color{black}#1}}
\usepackage{amssymb,amsfonts, amsmath} %
\usepackage{bm}
\usepackage{todonotes}
\usepackage{mathtools, relsize,geometry}
\usepackage{mathrsfs}
\usepackage{multirow}

\newcommand{\mat}[1]{\ensuremath{{\bm{{#1}}}}}
\makeatletter

\renewcommand{\vec}[1]{%
	\ifcat\relax\noexpand#1%
	\ensuremath{\boldsymbol{\lowercase{#1}}}%
	\else
	\ensuremath{\mathbf{\lowercase{#1}}}%
	\fi
}
\newcommand{\R}{\ensuremath{\mathbb{R}}}

\newcommand{\T}{\mathsmaller{\text{T}}}
\newcommand{\kron}{\otimes}
\newcommand{\Kr}{\mathscr K}
\newcommand{\bX}{\bm X}
\newcommand{\bR}{\bm R}
\newcommand{\bM}{\bm M}
\newcommand{\bN}{\bm N}
\newcommand{\bA}{\bm A}
\newcommand{\bE}{\bm E}
\newcommand{\bH}{\bm H}
\newcommand{\bC}{\bm C}

\newcommand{\bB}{\bm B}
\newcommand{\bD}{\bm D}
\newcommand{\bF}{\bm F}
\newcommand{\bG}{\bm G}
\newcommand{\bP}{\bm P}
\newcommand{\bY}{\bm Y}

\newcommand{\bZ}{\bm Z}
\newcommand{\sscg}{{{\sc s}{\sc \small s}--{\sc {cg}}}}
\newcommand{\sspcg}{{\sc s}{\sc \scriptsize s}--{\sc {cg}}}
\newcommand{\cg}{{\sc cg}}
\newcommand{\tpcg}{{\sc tpcg}}

\theoremstyle{remark}
\newtheorem{remark}[theorem]{Remark}
\newtheorem{example}[theorem]{example}

\usepackage{pgfplots}
\usepackage{graphicx,epstopdf} 
\usepackage{subfig}
\usepackage{tikz}
\usepackage{algorithm}
\usepackage{algpseudocode}
\usepackage{algorithmicx}

\usetikzlibrary{shapes.geometric}

\definecolor{matlabred}{rgb}{0.9047,    0.1918,    0.1988}
\definecolor{matlabblue}{rgb}{0.2941    0.5447    0.7494}
\definecolor{matlabgreen}{rgb}{	0.3718    0.7176    0.3612}
\definecolor{matlaborange}{rgb}{1.0000    0.5482    0.1000}

\usepackage{algpseudocode}
\algnewcommand{\LineComment}[1]{\Statex \(\%\) \small \textit{#1} \(\%\)}

\Crefname{ALC@unique}{Line}{Lines}
\usepackage{algorithm}

\title{A subspace-conjugate gradient method for linear matrix equations%
\thanks{Version of \today}}
\author{Davide Palitta\thanks{Dipartimento di Matematica and (AM)$^2$,
Alma Mater Studiorum Universit\`a di Bologna,
Piazza di Porta San Donato  5, I-40127 Bologna, Italy,
{\tt \{davide.palitta\}\{martina.iannacito\}\{valeria.simoncini\}@unibo.it}}
 \and Martina Iannacito$^\dagger$
 \and Valeria Simoncini$^\dagger$\thanks{IMATI-CNR, Pavia, Italy. }}

\ifpdf
\hypersetup{ pdftitle={Subspace-conjugate gradient} }
\fi

\begin{document}
\maketitle

\renewcommand{\thefootnote}{\fnsymbol{footnote}}
\maketitle \pagestyle{myheadings} \thispagestyle{plain}
\markboth{ D.\ PALITTA, M.\ IANNACITO, V.\ SIMONCINI}{SUBSPACE-CONJUGATE GRADIENT METHOD FOR MATRIX EQUATIONS}


\begin{abstract}
The efficient solution of large-scale
multiterm linear matrix equations is a challenging task 
in numerical linear algebra, and it is a largely open problem.
We propose a new iterative scheme for symmetric and positive
definite operators, significantly advancing methods such as truncated 
matrix-oriented Conjugate Gradients ({\sc cg}). The new algorithm 
capitalizes on the low-rank matrix format of its iterates by fully exploiting
the subspace information of the factors as iterations proceed.
The approach implicitly relies on orthogonality conditions imposed over
much larger subspaces than in {\sc cg}, unveiling insightful
 connections with subspace projection methods.
The new method is also equipped with memory-saving strategies.
In particular, we show that
for a given matrix $\bY$, the action ${\cal L}(\bY)$ in low rank format
 may not be evaluated exactly due to memory constraints.  
This problem is often underestimated,
though it will eventually produce Out-of-Memory breakdowns for a sufficiently large
number of terms. We propose an ad-hoc 
randomized range-finding strategy that appears to fully resolve this shortcoming.

Experimental results with typical application problems
illustrate the potential of our approach over various methods developed
in the recent literature.
\end{abstract}




\section{Introduction}
We are interested in the numerical solution of the problem
\begin{equation}\label{eqn:main}
\bA_1 \bX \bB_1 + \ldots + \bA_\ell \bX \bB_\ell = \bC,
\end{equation}
where $\bA_i\in\R^{n_A\times n_A}$, $\bB_i\in\R^{n_B\times n_B}$,
$i=1, \ldots, \ell$, are symmetric matrices,
and $\mat{C}\in\R^{n_A\times n_B}$ has rank $s_C\ll \min\{n_A, n_B\}$ so that we can write $\bC=C_1C_2^\T$, $C_1\in\mathbb{R}^{n_A\times s_C}$, $C_2\in\mathbb{R}^{n_B\times s_C}$.
{By introducing the linear operator}
${\cal L}(\bX) = \bA_1 \bX \bB_1 + \ldots + \bA_\ell \bX \bB_\ell$,
in the following we will use the more compact notation
$$
{\cal L}(\bX) = \bC
$$
for the matrix equation above. 
%
%
We assume that $\cal L$ is positive definite in the matrix inner product, that is
it holds that $\langle \bX, {\cal L}(\bX)\rangle >0$ for any nonzero
$\bX\in\mathbb{R}^{n_A\times n_B}$.
Given two $n\times m$ matrices ${\bm X}, {\bm Y}$, we define the matrix inner product as
$$
\langle {\bm X}, {\bm Y} \rangle = {\rm trace}({\bm X}^\T {\bm Y}).
$$
This inner product defines the Frobenius norm $\|\bX\|_F^2 =
\langle {\bm X}, {\bm X} \rangle$.

 Encountered samples of $\cal L$ include
for instance the generalized Lyapunov equation 
$\mathcal{L}(\mat{X}) = \mat{AXE^\T} + \mat{EXA}^{\T}$
and its multiterm counterpart $\mathcal{L}(\mat{X}) = \bA \bX\bE^\T  + \bE \bX \bA^{\T} + \mat{MXM}^{\T}$
with $\mat{A}, \mat{E}, \mat{M}\in\R^{n\times n}$, as they occur in control theory
\cite{Benner.Damm.11},\cite[Ch.6]{BCOW.17}; 
in our setting we have the additional hypothesis that all coefficient
matrices are symmetric.
In the following we will call a multiterm Lyapunov equation an equation as in~\eqref{eqn:main} where $\bC$ is symmetric and $\mathcal{L}$ is such that $(\mathcal{L}(\bX))^{\T}=\mathcal{L}(\bX)$ for any symmetric $\bX$. This 
implies that the solution to a multiterm Lyapunov equation is symmetric.

More general forms typically arise whenever {the terms on the left and right of the unknown} have different meaning in the original application,
such as geometric space vs time (see, e.g., \cite{Urbanetal.enumath.22},\cite{LOLI20202586}),
or geometric space vs
parameter space (e.g., \cite{Powelletal2017},\cite{Benner.Onwunta.Stoll.15}), or
optimization (e.g., \cite{Dolgov.Stoll.17},\cite{Stoll.Breiten.15}).
We will refer to this general latter structure as multiterm Sylvester equation\footnote{Often
the term ``generalized Sylvester'' is used for the same equation. The term
``generalized'' is also employed for two-term equations, for rectangular problems, and some
multi-variable contexts. To avoid ambiguity we prefer to use the name ``multiterm Sylvester''.}.

Many different solid methods for the solution of 
equation~\eqref{eqn:main} for $\ell= 2$ have been devised in the
past two decades (see, e.g., the survey~\cite{Simoncini2016}).
On the other hand, 
having a number of terms $\ell> 2$ in~\eqref{eqn:main} makes the numerical 
treatment of this equation extremely challenging. Fewer options are 
available in the literature for medium up to large dimensions of the
coefficient matrices. \rev{In particular, to the best of our knowledge, no decomposition-based method able to simultaneously triangularize $\ell>2$ matrices is available in the literature
so that, up to date, recasting the problem in terms of its Kronecker form is the only option to get a direct solution.
However, this strategy suffers from excessive
memory constraints and computational cost even for moderate dimensions of the coefficient matrices,
so that only iterative procedures for the solution of~\eqref{eqn:main} are being explored.}
Among the classes of contributions in this direction are
matrix-oriented Krylov methods with low-rank 
truncations~
\cite{KressnerTobler2011},%
\cite{Stoll.Breiten.15},\cite{SimonciniHao2023},\cite{PalittaKuerschner2021},
 projection methods tailored to the equation at 
hand~\cite{Jarlebringetal2018},\cite{SimonciniHao2023},\cite{Powelletal2017}, 
fixed-point iterations  \cite{Damm.08},\cite{Shanketal.16},
Riemannian optimization schemes~\cite{Biolietal2024}, and 
greedy procedures \cite{KressnerSirk2015}.

In the following we focus our attention on short recurrences associated
with matrix-oriented Krylov methods.
These schemes amount to adapting standard Krylov schemes for 
linear systems to matrix equations by leveraging 
the equivalence between~\eqref{eqn:main} and its Kronecker form. 
Thanks to our hypotheses on $\mathcal{L}$, the coefficient matrix of 
the linear system in Kronecker form is symmetric positive definite so that the 
Conjugate Gradient method (\cg) can be applied; 
see, e.g.,~\cite{KressnerTobler2011},\cite{SimonciniHao2023},\cite{BennerBreiten2013} and 
section~\ref{Truncated matrix-oriented CG} for more details.
By building upon the low rank structure of the right-hand side,
matrix-oriented \cg\ generates matrix recurrences, rather than
vector recurrences, in {\it factored} form, thus allowing high memory
and computational savings, while retaining the optimality properties
when brought back to the vectorized form. 
Unfortunately, as the iterations progress, recurrence factors may quickly
increase their rank, losing the advantages of the whole matrix-oriented procedure.
Rank truncation strategies of the factor iterates are usually enforced
so as to keep memory allocations under control. As a side effect, however,
convergence is often delayed, also possibly leading to stagnation 
\cite{KressnerTobler2011},\cite{Kressner.Plesinger.Tobler.14},\cite{SimonciniHao2023}.

By taking inspiration from this class of methods, we design a
new iterative scheme for the solution of~\eqref{eqn:main} that better
exploits the rich subspace information obtained
with the computed quantities, to define the next factorized iterates.
More precisely, at each iteration the next approximate solution and
direction are obtained by imposing a functional optimality with respect
to the whole range of the low rank factors available in the current iteration.
This should be compared with the approximate solution in matrix-oriented \cg,
obtained at each iteration by a functional optimality with respect to a single vector.
 
{The idea appears to be new, as} it goes far beyond
the algorithmic developments 
usually associated with matrix-oriented approaches. While being closer to
optimization procedures based on manifolds, it does not share the same complexity
in the definition of the funding recurrences, as the new method is still fully derived
from the original Conjugate Gradient method for linear systems.
Truncation strategies are devised to maintain the computed matrix iterates low rank.

In designing the new approach we address a memory allocation issue that becomes crucial
when the number $\ell$ of terms in (\ref{eqn:main}) is significantly larger than two.
More precisely, for a given matrix $\bY$, the action ${\cal L}(\bY)$ in low rank format
 may not be evaluated exactly, making it impossible to generate quantities such as
the residual matrix. This problem is often underestimated in the current literature, 
though it will eventually produce Out-of-Memory breakdowns in actual computations, for
$\ell$ large enough. We propose an ad-hoc 
randomized range-finding strategy that appears to fully resolve this shortcoming, 
keeping the memory allocations under control.

We name the new method the preconditioned {\it subspace-conjugate gradient method}
(\sscg)
to emphasize the role of subspaces in the recurrences.
This novel point of view leads to remarkable computational 
gains making \sscg\ a very competitive option for the solution of
multiterm linear matrix equations of the form~\eqref{eqn:main}.
Computational experiments with a selection of quite diverse problems illustrate
the potential of the new strategy, when compared with other methods specifically
designed for the considered matrix equations.

Here is a synopsis of the paper. After introducing the notation we use throughout the paper in section~\ref{Notation}, in section~\ref{Truncated matrix-oriented CG} we recall the matrix-oriented \cg\ method for~\eqref{eqn:main}. Section~\ref{sec:subspace-CG} sees the derivation of the subspace-conjugate gradient method, the main contribution of this paper, and in section~\ref{A first version  of the algorithm for the
multiterm Lyapunov equation} we illustrate a first pseudoalgorithm for multiterm Lyapunov equations. Some theoretical aspects of the novel procedure are studied in section~\ref{sec:discussion}. In section~\ref{sec:sylv} we generalize our method to the solution of multiterm Sylvester equations. 
We then present several memory- and time-saving
strategies:
we discuss low-rank truncations and
effective residual computation in section~\ref{sec:truncation}, 
and the employment of inexact coefficients in section~\ref{Inaccurate coefficients}.
Section~\ref{sec:precond} dwells with the inclusion of preconditioning strategies.
The resulting algorithm for multiterm Sylvester matrix equations
is summarized in section~\ref{sec:final_algo}.
Numerical results illustrating the competitiveness of our new procedure in solving multiterm matrix equations are presented in section~\ref{Numerical results}.
Conclusions are depicted in section~\ref{Conclusions}.
The Appendix collects some of the discussed algorithms.

\subsection{Notation}\label{Notation}
Throughout the paper, capital, bold letters ($\bX$) will denote $n_A\times n_B$ matrices,
 with capital letters ($X$) denoting their possibly low-rank factors, e.g., $\bX=X_1X_2^\T$. We already mention 
here that, for the sake of the presentation, we will often use the 
notation $\bX$ for our iterates, although 
we will operate with their low-rank factors only, without allocating
these matrices as full.
 See section~\ref{sec:subspace-CG} for further details.

Greek letters ($\alpha$) will denote scalars whereas bold Greek letters ($\bm\alpha$) 
will be used for small dimensional matrices.
Moreover, $\text{blkdiag}(\bm{\alpha}_1, \ldots,\bm{\alpha}_s)$ 
denotes the block diagonal matrix having on the diagonal the matrices 
$\bm{\alpha}_1, \ldots,\bm{\alpha}_s$.
The symbol $\otimes$ denotes the Kronecker product whereas $\text{vec}(\cdot)$ is the 
operator that stacks the columns of a matrix one below the other.
For a matrix $\bX$, $\text{range}(\bX)$ is the space spanned by the
columns of $\bX$. 


\section{Truncated matrix-oriented CG}\label{Truncated matrix-oriented CG}
This section is devoted to surveying the well exercised
matrix-oriented version of the Conjugate Gradient method, together
with its truncated variant. 

When addressing the solution of (\ref{eqn:main}), the 
Kronecker formulation of the problem, namely
\begin{equation}\label{eqn:kron}
(\bB_1^\T \otimes \bA_1  + \ldots + \bB_\ell^\T \otimes \bA_\ell )
 {\rm vec }(\bX)
= {\rm vec }(\bC) \quad\Leftrightarrow \quad
{\cal A} x = c,
\end{equation}
allows one to directly employ the classical Preconditioned
Conjugate Gradient ({\sc pcg}) method. A careful implementation should avoid the
explicit construction of $\cal A$, so that matrix-vector products
can be carried out in the original matrix form with the
${\cal L}$ operator.  
Even with this precaution, all vector iterates still
have $n_A n_B$ components, so that whenever
$n_A, n_B$ are large, memory allocations may become prohibitive.
A particularly convenient way out occurs when $\bC$ is very low rank.
In this case, under certain conditions, the \rev{exact} solution $\bX$ may also
be well approximated by a low rank matrix 
\cite{Beckermann.Kressner.Tobler.13},%
\cite{BennerBreiten2013},%
\cite{Grubisic.Kressner.14},%
\cite{KressnerTobler2011}.
To exploit this characterization, all vector iterates are transformed
back to matrices and kept in low-rank matrix format.
Unfortunately, although  during the first few {\sc pcg} iterations
the iterates maintain low rank, the rank itself grows  as the method
proceeds.
To control the memory requirements of the procedure after the first
few iterations, truncation of the iterate factors are usually performed.
{We refer to, e.g., 
\cite[Algorithm 2]{KressnerTobler2011}, \cite{SimonciniHao2023} for the
algorithmic description.}

As long as no low-rank truncations are performed, 
{truncated {\sc pcg}}
is mathematically equivalent to applying the 
standard, \emph{vectorized} \cg\ method to the linear system
stemming from~\eqref{eqn:main} via the Kronecker form in (\ref{eqn:kron}). 
Implementing low-rank truncations may be viewed as a simple computational device
to make the solution process affordable in terms of storage allocation and not 
as an algorithmic advance.
The final attainable accuracy when truncation is in place depends on the truncation
tolerance and on the decay of the singular values in the problem solution matrix; we
refer to  
\cite{Beckermann.Kressner.Tobler.13},\cite{Kressner.Plesinger.Tobler.14},\cite{SimonciniHao2023}
for a detailed discussion on the effect of truncation.

\section{The \sscg\ method for the multiterm Lyapunov equation}\label{sec:subspace-CG}
In this section we derive our new method for the multiterm Lyapunov equation,
that is we assume that ${\cal L}(\bX)={\cal L}(\bX)^\T$ for any symmetric $\bX$, and
that $\bC$ is symmetric. In section~\ref{sec:sylv} we will discuss the changes occurring
when generalizing the procedure to the multiterm Sylvester case.  

The key idea is to
build a Conjugate Gradient type method remaining in $\R^n$, 
while generating quantities based on {\it subspaces} of $\R^n$, rather than
on {\it vectors} of $\R^{n^2}$, the way the original {truncated {\sc pcg}} does.

We follow the classical derivation of \cg\ as a procedure to approximate
the minimizer  of a convex function. 
Let the function $\Phi:\R^{n\times n}\rightarrow\R$ be defined as
\begin{equation} \label{eq:Phi:X}
\Phi(\mat{X}) =
\frac 1 2 \langle \mat{X}, \mathcal{L}\bigl(\mat{X}\bigr)\rangle -
\langle \mat{X}, \mat{C} \rangle.
\end{equation}
We then consider the following minimization problem: Find $\mat{X}\in\R^{n\times n}$ such that
$$
\bX ={\rm arg} \min_{\mat{X}\in\R^{n\times n}}\Phi(\mat{X}).
$$
%
Starting with a zero initial guess $\mat{X}_0\in\R^{n\times n}$ and
${P}_0\in\R^{n\times s_C}$,  where we recall that $s_C$ is the
rank of $\bC$, we define the recurrence $\{\bX_k\}_{k\ge 0}$
of approximate solutions by means of the following relation
\begin{equation} \label{eq:update:X}
	\mat{X}_{k+1} = \mat{X}_k + {P}_k\vec{\alpha}_k{P}_k^{\T},
\end{equation}
where $\vec{\alpha}_k\in\R^{s_k\times s_k}$ and ${P}_k\in\R^{n\times s_k}$,
with corresponding recurrence for the residual
$\mat{R}_{k+1} = \bC - {\cal L}(\mat{X}_{k+1})$, that is
	$\mat{R}_{k+1} = \mat{R}_k - \mathcal{L}({P}_k\vec{\alpha}_k{P}_k^{\T})$.
We emphasize that $s_k$ depends on the iteration index $k$, that is the number of columns of $P_k$ may (and will)
change as the iterations proceed, possibly growing
up  to a certain maximum value, corresponding
to the maximum allowed rank of all iterates.
Since we assume that the operator $\cal L$ is
symmetric, that is ${\cal L}(\bX) = ({\cal L}(\bX))^\T$ for any symmetric matrix $\bX$,
and that $\bC=CC^\T$, all iteration matrices are square and symmetric.
In section~\ref{sec:sylv} we will relax these assumptions,
yielding possibly rectangular solution and iterates.

Like in the vector case, we require that the matrix $\bP_k=P_kP_k^{\T}$
satisfies a descent direction
 requirement completely conforming  to the vector case, that is
\begin{equation}\label{eqn:descent}
\langle\nabla\Phi(\mat{X}_k), \mat{P}_k\rangle < 0 .
\end{equation}
To determine $\vec{\alpha}_k$, we
let $\phi(\vec{\alpha}) = \Phi(\mat{X}_{k} + {P}_k\vec{\alpha}{P}_k^{\T})$.
For given $\mat{X}_{k}, {P}_k$,
at the $k$th iteration we construct $\vec{\alpha}_k$ so that
\begin{equation} \label{eq:min:alpha}
	\phi(\vec{\alpha}_k)=
\min_{\vec{\alpha}\in\R^{s_k\times s_k}}\phi(\vec{\alpha}) .
\end{equation}



The minimizer $\vec{\alpha}_k$ can be explicitly determined by
solving a linear matrix equation of reduced dimensions,
as
described in the following result.

\begin{proposition} \label{prop:alpha}
Assume that $\bA_i$ and $\bB_i$ are symmetric, and
that ${\cal L}$ is positive definite.
The minimizer $\vec{\alpha}_k\in\R^{s_k\times s_k}$ of~\eqref{eq:min:alpha}
is the unique solution of
	\begin{equation} \label{eq:prop:alpha:thesis}
 P_k^T {\cal L}(\mat{X}_k + P_k\vec{\alpha}P_k^T) P_k =
 P_k^T \mat{C} P_k,
        \end{equation}
or, equivalently, of
$P_k^T {\cal  L}(P_k\vec{\alpha}P_k^T) P_k = P_k^T \mat{R}_k P_k$.
\end{proposition}

\begin{proof}
We start by explicitly writing the function $\phi$, that is
$$\phi(\vec{\alpha}) =
\frac{1}{2} \langle \mat{X}_{k} + {P}_{k}\vec{\alpha}{P}_{k}^{\T},
{\cal L}(\mat{X}_{k} + {P}_{k}\vec{\alpha}{P}_{k}^{\T})\rangle -
\langle \mat{X}_{k} + {P}_{k}\vec{\alpha}{P}_{k}^{\T},\bC\rangle.
$$
To find the stationary points of $\phi$, we compute the
partial derivatives of $\phi$ with respect to $\vec{\alpha}$; this can
be done in matrix compact form; see, e.g.,
\cite{Petersen2008}.
We carry out this computation term by term, 
$$
\frac{\partial\text{tr}(\mat{X}_k^{\T}{\cal L}(\bX_k +
{P}_{k}\vec{\alpha} P_k^{\T}))}{\partial \vec{\alpha}}
=
\frac{\partial\text{tr}(\mat{X}_{{k}}^{\T}
{\cal L}({P}_{k}\vec{\alpha}{P}_{k}^{\T}))}{\partial \vec{\alpha}}
 = {P}_{k}^{\T}{\cal L}(\bX_k) {P}_{k},
$$
and
{\small
\begin{eqnarray*}
\frac{\partial \,\text{tr} (\bigl({P}_{k}\vec{\alpha}{P}_{k}^{\T})^\T
\mathcal{L}\bigl(\mat{X}_{k} + {P}_{k}\vec{\alpha}{P}_{k}^{\T}\bigr)\bigr)}{\partial \vec{\alpha}}&=&
\frac{\partial\, \text{tr}\,\bigl(\bigl({P}_{k}\vec{\alpha}{P}_{k}^{\T}\bigr)^{\T}\mathcal{L}\bigl(\mat{X}_{k}\bigr)\bigr)}{\partial \vec{\alpha}} +\frac{\partial\text{tr}\,\bigl(\bigl({P}_{k}\vec{\alpha}{P}_{k}^{\T}\bigr)^{\T}
	{\cal L}({P}_{k}\vec{\alpha}{P}_{k}^{\T})\bigr)}{\partial \vec{\alpha}} \\
&=&
	{P}_{k}^{\T}\mathcal{L}\bigl(\mat{X}_{k}\bigr){P}_{k} +
 2{P_k}^{\T}{\cal L}(P_k \vec{\alpha} P_k^\T)P_k.
\end{eqnarray*}
}
%
%
Moreover, it holds that
$\frac{\partial \text{tr}\,\bigl(\bigl({P}_k\vec{\alpha}{P}_k^{\T}\bigr)^{\T}\mat{C} \bigr)}{\partial \vec{\alpha}} = {P}_k^{\T}\mat{C}{P}_k$, see, e.g., \cite[Equations (101), (102), (108), (113)]{Petersen2008}.
The final expression of the Jacobian of $\phi$ with respect to $\vec{\alpha}$ is thus

$$\frac{\partial \phi(\vec{\alpha})}{\partial \vec{\alpha}}  =
P_k^\T {\cal L}(P_k \vec{\alpha} P_k^\T) P_k - {P}_{k}^{\T}\mat{R}_{k}{P}_{k}.
$$
Consequently, the solution $\vec{\alpha}_k$ of \eqref{eq:prop:alpha:thesis} is a
stationary point of $\phi$.
To ensure that $\vec{\alpha}_k$ is a minimizer, we show that the
Hessian of $\phi$ is positive definite.
To ease the reading, the Jacobian of $\phi$ is vectorized, resulting in
$$
\text{vec}\left(\frac{\partial \phi(\vec{\alpha})}{\partial \vec{\alpha}}\right)=
		\sum_{i=1}^{\ell}\bigl(P_k^\T\mat{B}_iP_k\kron P_k^\T\mat{A}_iP_k\bigr)\text{vec}\bigl(\vec{\alpha}\bigr) - \text{vec}\bigl({P}_{k}^{\T}\mat{R}_{k}{P}_{k}\bigr).
$$

The Hessian $\mat{H}_k\in\R^{s_k^2\times s_k^2}$ is given by
$\mat{H}_k= \sum_{i=1}^{\ell}\bigl(P_k^{\T}\mat{B}_iP_k\kron P_k^{\T}\mat{A}_iP_k\bigr)$.
Then, using the hypothesis on the operator $\cal L$,
for any nonzero $\vec{y}\in\R^{s_k^2}$, it holds that
$\vec{y}^\T\mat{H}_k\vec{y}>0$, so that $\mat{H}_k$ is positive definite,
and $\vec{\alpha}_k$ is a minimizer of $\phi$.
\end{proof}

\vskip  0.05in
\begin{remark}\label{RemarkP}
{\rm 
For $P_k$ full rank, the  
quantity $P_k \vec{\alpha}_k P_k^\T$ is invariant with respect to the
basis of ${\rm range}(P_k)$ used to compute $\vec{\alpha}_k$.  $\hfill\square$
}
\end{remark}
\vskip  0.05in

The minimization problem in \eqref{eq:min:alpha} can also be recast in terms of an orthogonality condition. 
Indeed, solving~\eqref{eq:min:alpha} is equivalent to imposing 
the following \textit{subspace orthogonality condition}
\begin{equation}
	\label{eq:Rk1orthPk}
	\text{vec}(\mat{R}_{k+1})\perp\text{range}({P}_k\kron{P}_k).
\end{equation}
{More precisely,~\eqref{eq:Rk1orthPk} is equivalent to 
$P_k^\T \mat{R}_{k+1} P_k =0$ with
$\mat{R}_{k+1}=\mat{C} - \mathcal{L}(\mat{X}_k+{P}_k\vec{\alpha}{P}_k^{\T})$.}
Hence, the computation of $\vec{\alpha}_k$ follows from a {\it local}
matrix Galerkin projection of the original problem onto a space of
dimension $s_k^2$ given by the current direction matrix factor $P_k$. We will return on
this aspect in section \ref{sec:discussion}.

\vskip 0.05in
\begin{remark}\label{rem:gal}
{\rm
Thanks to the linearity of the matrix operator $\cal L$, products of
the form $P_k^\T {\cal L}(\bX_k  + \bY_k)P_k$ for some matrix $\bY_k$, become
$$
P_k^\T {\cal L}(\bX_k  + \bY_k)P_k =
P_k^\T {\cal L}(\bX_k)P_k  + P_k^\T{\cal L}(\bY_k)P_k.
$$
 In addition,
using the low rank factor form of the argument matrix, the 
left and right products {act} on  the coefficient matrices
as in reduction processes; 
see, e.g., \cite{Antoulas.05,Simoncini2016}.
For instance, for $\bY_k=P_k {\bm \omega}_k P_k^\T$ and substituting the
general operator $\cal  L$  in (\ref{eqn:main}),  we obtain
$$
{P}_k^{\T}{\cal L}({P}_k\vec{\omega}_k{P}_k^{\T}){P}_k=
\widetilde\bA_1 {\bm \omega}_k \widetilde \bB_1 + \ldots + 
\widetilde\bA_\ell {\bm \omega}_k \widetilde \bB_\ell,
$$
with $\widetilde \bA_i= {P}_k^{\T} \mat{A}_i{P}_k$,
and $\widetilde \bB_i= {P}_k^{\T} \mat{B}_i{P}_k$,
for $i=1, \ldots, \ell$.
$\hfill\square$
}
\end{remark}
\vskip 0.05in

We define the recurrence for
the directions $\mat{P}_k$ as
$$	\mat{P}_{k+1} = \mat{R}_{k+1} + {P}_k\vec{\beta}_k{P}_k^{\T}.
$$
The matrix $\vec{\beta}_{k}\in\R^{s_k\times s_k}$
is obtained by imposing that the new directions $\bP_{k+1}$ are $\mathcal{L}$-orthogonal with respect to the previous ones. In particular, we write
$\text{vec}(\mat{P}_{k+1}) \perp_{\mathcal{L}} \text{range} ({P}_k\kron{P}_k)$,
that is
\begin{equation} \label{eq:beta:subsp-orth}
	({P}_k\kron{P}_k)^{\T}\text{vec}(\mathcal{L}(\mat{P}_{k+1})) = 0.
\end{equation}
Inserting the expression for $\mat{P}_{k+1}$, (\ref{eq:beta:subsp-orth}) becomes
${P}^{\T}_k \mathcal{L}(\mat{R}_{k+1} + {P}_k\vec{\beta}_k{P}_k^{\T}) {P}_k = 0$,
that is,
$$
{P}_k^{\T}\mathcal{L}(\mat{R}_{k+1}){P}_{k}  +
{P}_k^{\T}\mathcal{L}(P_k\vec{\beta}_k P_k^\T){P}_{k} = 0.
$$
This is a linear matrix equation in the unknown $\vec{\beta}_k$, with
the same coefficient matrices of the linear matrix equation used to
compute $\vec{\alpha}_k$.
Once again, and using the considerations in Remark~\ref{rem:gal},
$\vec{\beta}_k$ is obtained
by solving a linear matrix equation of the same type as the original one, but
with very small dimensions, by projecting the problem orthogonally onto $\text{range}(P_k\otimes P_k)$,
in a matrix sense.

Concerning the quality of the computed direction iterates,
we next show that the descent direction property (\ref{eqn:descent})
 is maintained.

\begin{proposition}\label{Prop:Discent_direction}
Let $\mat{P}_{k+1}= {P}_{k+1} \vec{\gamma}_{k+1} P^{\T}_{k+1}$ for
some matrix $\vec{\gamma}_{k+1}$.
Then	${\bP}_{k+1}$ is a descent direction.
\end{proposition}

\begin{proof}
	To prove that $\bP_{k+1}$ is a descent direction matrix,
we must show that
\begin{equation} \label{proof:Pk1:eq:innprod}
   \langle \nabla\Phi(\mat{X}_{k+1}), \mat{P}_{k+1}\rangle < 0 ,
\end{equation}
with $\Phi$ defined in \eqref{eq:Phi:X}.
Following~\cite[Equations (101), (102), (108), (113)]{Petersen2008}, we compute the terms of the Jacobian of $\Phi$ with respect to $\mat{X}_{k+1}$, yielding
$$\frac{\partial \text{tr}(\mat{X}_{k+1}^\T\mat{C})}{\partial\mat{X}_{k+1}} =
{\mat{C}}, \qquad
\frac{\partial \text{tr}
\bigl(\mat{X}_{k+1}^\T\mathcal{L}(\mat{X}_{k+1})\bigr)}{\partial\mat{X}_{k+1}} =
2\mathcal{L}(\mat{X}_{k+1});
$$
Here we used the fact that $\bA_i, \bB_i$ are symmetric.
Hence,
$\nabla\Phi(\mat{X}_{k+1}) = \mathcal{L}(\mat{X}_{k+1})
- \mat{C} = -\mat{R}_{k+1}$.
The inner product of \eqref{proof:Pk1:eq:innprod} can be written as
\begin{equation*}
\begin{split}
\langle \nabla\Phi(\mat{X}_{k+1}), \mat{P}_{k+1}\rangle &
= \langle -\mat{R}_{k+1},  \mat{R}_{k+1} +
{P}_{k}\vec{\beta}_k{P}^{\T}_{k}\rangle\\
& = - \|\mat{R}_{k+1}\|_F^2 -\langle \mat{R}_{k+1},
{P}_{k}\vec{\beta}_k{P}^{\T}_{k}\rangle = - \|\mat{R}_{k+1}\|_F^2 < 0 ,
\end{split}
\end{equation*}
where, by using \eqref{eq:Rk1orthPk}, we have that
$\langle \mat{R}_{k+1}, {P}_{k}\vec{\beta}_k{P}^{\T}_{k}\rangle =  0$.
\end{proof}

The proof above relies on the property
$\langle \mat{R}_{k+1}, \mat{P}_{k}\rangle = 0$.
In our setting, this annihilation is ensured
 in a stronger sense than in the matrix-oriented \cg\ algorithm.
More precisely, 
not only $\text{vec}(\mat{P}_{k})^\T \text{vec}(\mat{R}_{k+1})=0$ holds, which would be enough to show Proposition~\ref{Prop:Discent_direction},
but the stronger constraint
$P_k^{\T}\mat{R}_{k+1}P_k=0$ holds.
This \emph{block} orthogonality is reminiscent of block methods for
multiple right-hand side systems~\cite{blockCG}, though in practice there
are no further connections.  


\begin{algorithm}[t]
{\footnotesize
\begin{algorithmic}[1]
\smallskip
\Statex \textbf{Input:} Operator $\mathcal L:\mathbb{R}^{n\times n}\rightarrow \mathbb{R}^{n\times n}$, right-hand side $\bC$, initial guess $\bX_0$, maximum number of iterations $\texttt{maxit}$, tolerance $\texttt{tol}$.
\Statex \textbf{Output:} Approximate solution $\bX_k$  such that $\|\mathcal L(\bX_k)-\bC\|\leq \|\bC\| \cdot \texttt{tol}$
\smallskip

\State Set $\bR_0=\bC-\mathcal{L}(\bX_0)$, $\bP_0=\bR_0=P_0P_0^{\T}$ 
\For{$k=0,\ldots,\mathtt{maxit}$}
\State Compute $\vec{\alpha}_k$ by solving $P_k^\T {\cal L}(P_k \vec{\alpha}_k P_k^\T)P_k = P_k^\T \bR_k P_k$\label{line:compute_alpha}
\State Set $\bX_{k+1}=\bX_k+P_k\vec{\alpha}_k P^{\T}_k$ in a factorized fashion $X_{k+1}\bm{\tau}_{k+1}X_{k+1}^{\T}=\bX_{k+1}$
 {\flushright{\hspace{8cm} optional: low-rank truncation of $\bX_{k+1}$}}
\State Set $\bR_{k+1}=\bC-\mathcal{L}(X_{k+1}\bm{\tau}_{k+1}X_{k+1}^{\T})$ in a factorized fashion $R_{k+1}\bm{\rho}_{k+1}R_{k+1}^{\T}=\bR_{k+1}$
 {\flushright{\hspace{8cm} optional: low-rank truncation of $\bR_{k+1}$}}
\If{$\|\bR_{k+1}\|\leq \|\bC\| \cdot \texttt{tol}$}
\State Return $\bX_{k+1}$
\EndIf
\State Compute $\vec{\beta}_k$ by solving $P_k^\T {\cal L}(P_k \vec{\beta}_k P_k^\T)P_k = -P_k^\T {\cal L}(\bR_{k+1}) P_k$\label{line:compute_beta}
\State Set $\bP_{k+1}=\bR_{k+1}+P_k\vec{\beta}_k P_k^{\T}$ in a factorized fashion $P_{k+1}\vec{\gamma}_{k+1}P_{k+1}^{\T}=\bP_{k+1}$
 {\flushright{\hspace{8cm} optional: low-rank truncation of $\bP_{k+1}$}}\label{alg:linePk1}
\EndFor
\State Return $\bX_{k+1}$
\end{algorithmic}    \caption{\sscg\ - vanilla version for multiterm Lyapunov equations.  \label{alg:subspaceCG} }
}
\end{algorithm}

%
%
%
%
%
%
%
%
%
%
%
%
%
%

\subsection{A first version  of the algorithm for the 
multiterm Lyapunov equation}\label{A first version  of the algorithm for the
multiterm Lyapunov equation}
Summarizing the previous derivation, 
the iteration of the \sscg\ scheme is
given in Algorithm~\ref{alg:subspaceCG}.
This algorithm includes extra commands with respect to our initial
presentation, which require more detailed explanation.
Following standard procedures,
the next iterates $\bX_{k+1}, \bP_{k+1}$, and $\bR_{k+1}$ are not explicitly computed, as this would lead to
storing large dense matrices. Each of these matrices is kept in factored form, whose
rank is truncated if necessary. The updating step is linked to the subsequent factorization step as follows.
Consider the approximate solution update, starting from
$\bX_{k}= X_{k} \vec{\tau}_{k} X_{k}^\T$. We write
$$
\bX_{k+1}=\bX_k + P_k \vec{\alpha}_k P_k^\T
= [X_k, P_k ] {\rm blkdiag}( \vec{\tau}_k, \vec{\alpha}_k)
 [X_k, P_k ]^\T =
 X_{k+1} \vec{\tau}_{k+1} X_{k+1}^\T,
$$
where $X_{k+1}$ is obtained as the reduced orthonormal factor of the QR decomposition of
$[X_k, P_k ]$, that is $[X_k, P_k ]=Q\bm{r}$, and $\vec{\tau}_{k+1} =
\bm{r} {\rm blkdiag}( \vec{\tau}_k, \vec{\alpha}_k) \bm{r}^\T$. A more precise implementation ensures
that $\vec{\tau}_{k+1}$ has full rank via an eigenvalue decomposition, that may lower the
rank of the  factor $X_{k+1}$. 
From a memory point of view, none of the full matrices in {bold
is stored, as factors are immediately created and saved.}
%
More details on this {rank reduction} 
will be given in section~\ref{sec:truncation}.

We stress that the updated terms $X_{k+1},
P_{k+1}$, and $R_{k+1}$ in Algorithm~\ref{alg:subspaceCG}
 each have orthonormal columns, thus simplifying some of the computations.
We also observe that the factor
$\vec{\gamma}_{k+1}$ in $P_{k+1}\vec{\gamma}_{k+1}P_{k+1}^{\T}$ does not play a role in later computations, as
only the subspace basis $P_{k+1}\otimes P_{k+1}$ is employed; 
{see Remark~\ref{RemarkP}}.
We postpone the complete implementation of the method
to section~\ref{sec:final_algo}, 
{after 
the description} of several advanced implementation strategies.

\section{Discussion on the developed procedure}\label{sec:discussion}

It is natural to compare the new \sscg\ with the standard matrix-oriented \cg.
The subspaces {acting in the \sscg\ method
are significantly larger than in matrix-oriented \cg, as explained next.}

\vskip 0.05in
\begin{remark}
{\rm 
In~\eqref{eq:Rk1orthPk}, orthogonality is imposed with respect to a subspace of 
$\R^{n^2}$ of dimension $s_k^2$.
On the other hand, in the matrix-oriented \cg\
condition~\eqref{eq:Rk1orthPk} is replaced by
$ \text{vec}(\bP_k)^{\T}
\text{vec}\bigl(\mat{R}_{k} - \alpha_k \mathcal{L}({P}_k {P}_k^{\T})\bigr)=0$,
with $\alpha_k\in\R$,
so that the orthogonality is imposed with respect to a subspace of $\R^{n^2}$ of dimension $1$.
Analogously, in~\eqref{eq:beta:subsp-orth}, orthogonality is imposed with respect
to a subspace of $\R^{n^2}$ of size $s_k^2$, whereas
in the matrix-oriented \cg, the orthogonality condition
is instead given by
$\text{vec}(\mat{P}_k)^{\T}\text{vec}\bigl(\mathcal{L}(\mat{P}_{k+1})\bigr)=0$,
{that is,} with respect to
a subspace of $\R^{n^2}$ of dimension~$1$.
$\hfill\square$ }
\end{remark}
\vskip 0.05in

The orthogonality conditions imposed in  deriving
the coefficient matrices $\vec{\alpha}_k$, $\vec{\beta}_k$ allow us
to extend orthogonality properties to other iterates, 
and to derive optimality results later in the section.
 To this end, we introduce
some notation for operations with the generalized Lyapunov operator.

For $R\in\R^{n\times s}$ we will denote with $\bA_\star \bullet  R$ the matrix
$$
\bA_\star \bullet R =[\bA_1R, \ldots, \bA_\ell R],
$$
and analogously for $\bB_\star\bullet  R$. Moreover, for $k\ge 0$ we  define
$$
\bA_\star^{k+1} \bullet R=
\bA_\star \bullet (\bA_\star^k \bullet R).
$$
We are going to characterize the spaces generated by the
factors $P_k, R_k$ of $\bP_k, \bR_k$, respectively, with the next proposition.
To this end, with the new notation we define the approximation space
$$
\Kr_k(\bA_\star,R_0)= {\rm range}([R_0, \bA_\star \bullet R_0, \ldots, \bA_\star^k \bullet R_0]);
$$
{Note that the spaces are nested, that is
$\Kr_k(\bA_\star,R_0)\subseteq \Kr_{k+1}(\bA_\star,R_0)$.
The notation above} is reminiscent of a block Krylov subspace. However, the space is
in general very different. 
{Indeed,}
it involves all matrices associated
with the operation $\bullet$, that is $\bA_1, \ldots, \bA_\ell$. 
{Although} the space dimension grows very quickly,
it can be significantly smaller than the sum of the number of
terms; {for instance, if} one of the $\bA_i$'s is the identity matrix,
then the product $\bA_\star \bullet  R_0$ will surely contribute at most
$(\ell-1) \cdot s$ vectors to the space 
${\rm range}([R_0, \bA_\star \bullet R_0])$, for $R_0\in\R^{n\times s}$,
due to the redundancy of $R_0$.
We also observe that for the Lyapunov operator,
$\Kr_k(\bA_\star,R_0)= \Kr_k(\bB_\star,R_0)$.
Note that the fact that the left and right spaces are the
same justifies our use of iterates in the form 
$P_k \vec{\omega} P_k^\T$ for some $\vec{\omega}$.

\begin{proposition}
\!\!\!\! Assume $\bX_0=0$ so that $R_0=C$. \!Then
${\rm range}(R_k),{\rm range}(P_k) \subseteq \Kr_k(\bA_\star,R_0)$.
\end{proposition}

\begin{proof}
For brevity, we denote range$(Y)$  as $r(Y)$. 
We also recall that the updates have the form
\begin{eqnarray*}
\bX_{k+1}&=&X_k{\bm\tau}_kX_k^\T + P_k{\bm\alpha}_kP_k^\T = [X_k, P_k] {\bm\tau}_{k+1} [X_k, P_k]^\T, \\
\bR_{k+1}&=&[R_0,  \bA_\star \bullet X_{k+1}] {\bm\rho}_{k+1} [R_0,  \bB_\star \bullet X_{k+1}]^\T,
\end{eqnarray*}
for some ${\bm\rho}_{k+1}$, and
$\bP_{k+1}=R_{k+1}{\bm\rho}_{k+1}R_{k+1}^\T + P_k{\bm\beta}_kP_k^\T 
= [R_{k+1}, P_k] {\rm blkdiag}({\bm\rho}_{k+1},{\bm\beta}_k) [R_{k+1}, P_k]^\T$.
It suffices to collect and write down the block components for the first few iterations.
The result then follows by induction.
Indeed,
\begin{eqnarray*}
P_0&=&R_0,  \quad X_1=P_0, 
    \quad r(R_1)\subset r([R_0, \bA_\star \bullet R_0])= \Kr_1\\
r(P_1)&\subset &r([R_1, R_0 ])\subset \Kr_1, \quad X_2\subset 
r([X_1, P_1])\subset r( [P_0, P_1])\subset r([R_0, R_1])\\
 &\quad& \quad r(R_2)= r([R_0, \bA_\star \bullet X_2])\subset
r([R_0, \bA_\star \bullet R_0, \bA_\star \bullet P_1])\subset r([R_0, \bA_\star \bullet R_0, \bA_\star^2 \bullet R_0])
= \Kr_2 \\
r(P_2)&\subset &r([R_2, P_1 ]) \subset r([R_0,R_1,R_2])\subset \Kr_2, 
\end{eqnarray*}
and so on.
\end{proof}

We proceed with a result ensuring that
subsequent residual matrices are block orthogonal to each other.
In the following we say that a matrix with blocks has maximum possible rank
if rank reduction is only due to linear dependence in exact arithmetic.
For instance, $[v, \bA_1 v, v]$ and $[v, \bA_1 v]$ have the same
maximum possible rank two.
As a related concept, we shall talk about 
maximum possible dimension for the subspaces generated by matrices with
the same maximum possible rank.

\begin{proposition}
For any $k>0$, let $\bR_k = R_k \vec{\rho}_k R_k^{\rm\T}$.
Assume that all updates have maximum possible rank, so that
${\rm range}(\bP_k)$, ${\rm range}(\bR_k)$, and $\Kr_k(\bA_\star,R_0)$ have the same dimension.
Then $R_k^{\rm\T} \bR_{k+1} R_k=0$.
\end{proposition}

\begin{proof}
{Let the columns of $U_k$ form an orthonormal basis for $\Kr_k$.}
%
Using the stated hypotheses, we have that $P_k=U_k G_1$ and
$R_k=U_k G_2$ with $G_1$, $G_2$ having full row rank.
From (\ref{eq:Rk1orthPk}) we have that
$0= P_k^{\rm\T} \bR_{k+1} P_k= G_1^\T U_k^\T \bR_{k+1} U_k G_1$,
and since $G_1$ has full row rank, it holds that $ U_k^\T \bR_{k+1} U_k=0$.
Since 
$R_k^{\rm\T} \bR_{k+1} R_k= G_2^\T U_k^\T \bR_{k+1} U_k G_2$,
the result follows.
\end{proof}

From the proof above, and under the same hypothesis of 
equal maximum possible dimension of 
${\rm range}(\bP_k)$, ${\rm range}(\bR_k)$ and $\Kr_k(\bA_\star,R_0)$,
it also follows that
$R_j^\T \bR_{k+1} R_j = 0$, $j=1, \ldots, k$.

We can next state a finite termination result.

\begin{proposition}
Assume that ${\rm range}({R}_k)= {\rm range}({P}_k)$
and have maximum possible dimension.
If $\cal L$ is the {multiterm} Lyapunov operator and it holds that
$$
\text{range}({\cal L}(P_k\vec{\alpha}_kP_k^\T)) \subseteq {\rm range}(P_k), 
$$
then the space range$(P_k\otimes P_k)$ contains the exact solution.
\end{proposition}

\begin{proof}
Under the stated hypothesis, 
${\cal L}(P_k\vec{\alpha}_kP_k^\T) = P_k \vec{\omega}_k P_k^\T$
for some matrix $\vec{\omega}_k$.
Hence,  $\bR_{k+1} = \bR_{k} -  P_k \vec{\omega}_k P_k^\T$ so
that ${\rm range}(R_{k+1})\subset {\rm range}(P_k)$.
From (\ref{eq:Rk1orthPk}) we have that
 $\bR_{k+1} \perp {range}(P_k)$, hence it must be  $\bR_{k+1}=0$.
\end{proof}

The formalization in terms of the space $\Kr_k$ allows us 
to characterize the new method with respect to less close but
still  known approaches.  Unless truncation takes place,
it holds that
range$(P_{k-1})\subseteq {\rm range}(P_{k})$, so that the iterate
$\mat{X}_{k+1}$ could be
written as $\mat{X}_{k+1}=P_{k} \vec{\tau}_{k} P_{k}^\T$, 
for some $\vec{\tau}_{k}$.
Moreover, the residual matrix $\bR_{k+1}$ is orthogonal, in the
matrix inner product, to the space  $\Kr_k\otimes \Kr_k$. These
two properties together show that under  maximum possible rank of
the iterates,
the new algorithm is mathematically equivalent to the Galerkin method 
for Lyapunov equations on the
subspace $\Kr_k(\bA_\star, R_0)$ \cite{Simoncini2016}.
For the operator 
${\cal L}(\bX)=\bA \bX + \bX \bA^\T +\bM \bX \bM^\T$, it is interesting  to
observe that $\Kr_k$ is the same as the space introduced in
 \cite[Section 4]{SimonciniHao2023}, although in there the space
was generated 
one vector at the time.
Moreover, the approach we are taking here  allows us
to update the iterates, rather than solving the projected system 
from scratch at each iteration.
The two approaches significantly deviate when truncation
takes place. 

By following the discussion in~\cite[Section 2.2]{Vietnam}, thanks 
to the orthogonality condition~\eqref{eq:Rk1orthPk} imposed to 
compute $\vec{\alpha}_k$, we can also state the following optimality result.

\begin{proposition}
Let $\bX$ be the exact solution 
to~\eqref{eqn:main} and $\|\bY\|^2_{\mathcal{L}}=\langle \bY,\mathcal{L}(\bY)\rangle$.
 Assume that $P_k\in\mathbb{R}^{n\times s_k}$ is computed by Algorithm~\ref{alg:subspaceCG} with no low-rank
truncation and that $\text{range}(P_k)$ has the maximum possible dimension. 
Then, $\bX_k=P_k\vec{\tau}_kP_k^\T$ is such that
$$
{\displaystyle \bX_k={\rm arg}\min_{\substack{\bZ=P_k\vec{\tau}P_k^\T
\\ \vec{\tau}\in\mathbb{R}^{s_k\times s_k}}}\|\bX-\bZ\|_{\mathcal{L}}.}
$$
\end{proposition}

\begin{proof}
 The proof follows the same lines as the proof of~\cite[Proposition 1]{Vietnam}.
\end{proof}

We end this section with a consideration on the numerical rank 
of the approximate solution iterate. Numerical experiments with matrix-oriented
\cg\ have shown that without truncation, the approximate solution rank tends to
significantly increase before decreasing towards its final value, corresponding
to the rank of the exact solution; see, e.g.,
 \cite{KressnerTobler2011}. Allowing for a richer linear combination of
the generated space columns, we expect that the approximate solution of
\sscg, with no truncation, will reach the final rank without an intermediate
growth. Numerical experiments seemed to confirm this fact, although a rigorous
analysis remains an open problem.

\section{The iteration for the multiterm Sylvester equation}\label{sec:sylv}
When the matrix operator ${\cal L}$ is nonsymmetric, that is
${\cal L}(\bX)\ne ({\cal L}(\bX))^\T$ for symmetric $\bX$, or $\bC$ is
indefinite or even nonsymmetric, the iteration
obtained with the new algorithm needs to be revised to address the general
Sylvester problem in (\ref{eqn:main}).
For general multiterm Sylvester equations, we still assume that all coefficient
matrices are symmetric, although the $\bA_i$'s and the $\bB_i$'s matrices are different.
This is in fact  the more common situation
for the problem (\ref{eqn:main}), as the coefficient matrices $\bA_i, \bB_i$,
$i=1, \ldots, \ell$ may even have different dimensions, leading to a 
rectangular solution matrix $\bX$.
Fortunately, the algorithmic differences are only technical, mostly affecting the notation.
Given two matrices $P_k^{l}, P_k^{r}$, the  iterates are computed by
means of the relation
$\bX_{k+1}=\bX_{k} + P_k^{l} \vec{\alpha}_k (P_k^{r})^\T$,
with $\bX_{k}\in\R^{n_A\times n_B}$,
and $\vec{\alpha}_k$ is computed by solving the reduced equation
$$
(P_k^{l})^\T {\cal L}(P_k^{l} \vec{\alpha} (P_k^{r})^\T) P_k^{r}
= (P_k^{l})^\T \mat{R}_k P_k^{r}.
$$
Analogously,
$\bP_{k+1}=\bR_{k+1} + P_k^{l} \vec{\beta}_k (P_k^{r})^\T$, so that
$P_{k+1}^{l} \vec{\gamma}_{k+1} (P_{k+1}^{r})^\T = \bP_{k+1}$,
where $\vec{\beta}_k$ solves
$(P_k^{l})^\T {\cal L}(\bR_{k+1}) P_k^{r}
+ (P_k^{l})^\T {\cal L}(P_k^{l} \vec{\beta} (P_k^{r})^\T) P_k^{r}=0$.
This construction of $\vec{\alpha}_k$ and $\vec{\beta}_k$ ensures
that the orthogonality conditions discussed in the previous sections
continue to hold, with respect to the left and right factors.

Like for the direction matrix $\bP_k$, also the iterates $\bX_k$ and $\bR_k$ 
will be nonsymmetric and possibly rectangular,
and they need to be factorized accordingly, 
namely $\bX_k=X_{k}^{l} \vec{\tau}_{k} (X_{k}^{r})^\T$ 
and $\bR_k=R_{k}^{l} \vec{\rho}_{k} (R_{k}^{r})^\T$.
All truncation strategies will have to
keep into account the nonsymmetry of the iterate, and in particular, we
will see that all computations need to be performed in a mirrored fashion
for the left and right spaces.
The overall algorithm for the resulting
multiterm Sylvester equation will be given in section~\ref{sec:final_algo},
Algorithm~\ref{alg:alphaCG}.

\section{Advanced implementation devices}\label{Enhancement strategies}
In this section we discuss some advanced
devices 
to make  Algorithm~\ref{alg:subspaceCG} competitive and robust
in terms of memory requirements,
running time, and final attainable accuracy for the given chosen truncation thresholds.
We start by analyzing in detail the low-rank iterate truncation, including
the residual matrix computation, 
then we discuss the computation of the (matrix) coefficients 
$\vec{\alpha}_k$ and $\vec{\beta}_k$.


\subsection{Low-rank truncations}\label{sec:truncation}
Solvers for large-scale matrix equations often require a low-rank truncation step to make the overall solution process affordable in terms of storage allocation. 
This is usually carried out by performing a thin QR decomposition 
and a subsequent truncated singular value decomposition (SVD) 
of small dimensional objects; see, e.g.,~\cite{KressnerTobler2011}.
For the sake of brevity, we will refer to this procedure as a QR-SVD (low-rank) truncation in the following.
\rev{In matrix-oriented constructions of \cg\ type methods, the truncation step is
a crucial part of the implementation, because it determines the actual success of the whole
procedure. A well designed truncation strategy, balancing the low-rank requirement and the
singular value accuracy, may allow the algorithm to deliver a sufficiently accurate
solution. In the past few years the effects of truncation have been analyzed -- both
experimentally and theoretically -- in several articles, see, e.g.,
 \cite{KressnerTobler2011},\cite{Kressner.Plesinger.Tobler.14},\cite{SimonciniHao2023},\cite{Stoll.Breiten.15}.
In our setting the space truncation is completely analogous to that encountered
in truncated \cg, so that a similar sensitivity to truncation is expected;
 this was confirmed by our extensive computational experience;
we refer to Example~\ref{ex:diffreac} for a sample.}

In our setting, if we consider
 $\bX_{k}= X_{k}^l \vec{\tau}_{k} (X_{k}^r)^\T$, this is updated as  
$$
\bX_{k+1}=\bX_k + P_k^l \vec{\alpha}_k (P_k^r)^\T
= [X_k^l, P_k^l ] {\rm blkdiag}( \vec{\tau}_k, \vec{\alpha}_k)
 [X_k^r, P_k^r ]^\T.
$$
Let $Q^l{\bm r}^l= [X^l_k, P_k^l]$, $Q^r{\bm r}^r= [X^r_k, P_k^r]$ be 
the thin QR decompositions of the two matrices, 
and compute the singular value decomposition 
${\bm  r}^l {\rm blkdiag}( \vec{\tau}_k, \vec{\alpha}_k) ({\bm r}^r)^{\T}=
U\Sigma V^{\T}$, with 
$\Sigma={\rm diag}(\sigma_1, \ldots, \sigma_s)$.
The low-rank truncation then takes place following two different criteria. 
The first one uses a threshold {\tt tolrank} and the other one 
 a maximum rank {\tt maxrank}.  In the first case, 
the number of columns $\widehat j_{k+1}$ of the low-rank 
terms $X_{k+1}^l$ and $X_{k+1}^r$ defining $\bX_{k+1}$ will be given by
$\widehat j_{k+1}={\rm arg}\max_j\{\sigma_j\, : \,(\sigma_j/\sigma_1)\leq\mathtt{tolrank}\}$, where the $\sigma_j$s are the
singular values just computed.
In the second case, for $\Sigma$ of size $s \times s$, we have
$\widehat j_{k+1}= \min \{ \mathtt{maxrank}, s\}$.
These two selections of $\widehat j_{k+1}$ are often performed in different moments of the iterative solver. 
The {\tt tolrank} criterion is preferable at an initial stage, when the iterates rank
is still moderate by construction. In later iterations, 
memory constraints generally force the application of the more 
aggressive truncation based on {\tt maxrank}. 
An automatic switch between the two truncation policies  is  obtained 
as follows
$$
\widehat j_{k+1}= \min \{ \mathtt{maxrank}, s,{\rm arg}\max_j\{\sigma_j\, : \,(\sigma_j/\sigma_1)\leq\mathtt{tolrank}\}\}.
$$
Once $\widehat j_{k+1}$ is selected, we define 
$X_{k+1}^l=Q^lU_{1:\widehat j_{k+1}}$, $X_{k+1}^r=Q^rV_{1:\widehat j_{k+1}}$, 
and $\bm{\tau}_{k+1}=\Sigma_{1:\widehat j_{k+1}}$, with the previous notation.
%
%
{In the rest of the paper, we will adopt the notation (see, e.g., \cite{KressnerTobler2011})
\begin{equation}\label{eq:QRSVD}
[X_{k+1}^l,\bm{\tau}_{k+1},X_{k+1}^r] = {\cal T}([X_{k}^l,P_{k}^l],\text{blkdiag}(\bm{\tau}_{k},\bm{\alpha}_{k}),[X_{k}^r,P_{k}^r],\mathtt{params}),
\end{equation}
for the computation of the QR-SVD truncated updating. In~\eqref{eq:QRSVD}, {\tt params} is a shorthand notation that indicates that all the necessary parameters are given in input. In particular, our QR-SVD requires the values
{\tt tolrank} and {\tt maxrank}.}
A QR-SVD truncation can be applied to compute $\bP_{k+1}=P_{k+1}^l\bm{\gamma}_{k+1}(P_{k+1}^r)^{\T}$ as well.

In principle, the term 
$\bR_{k+1}=R_{k+1}^l\vec{\rho}_{k+1}(R_{k+1}^r)^{\T}$ could be computed in the same way. However,
we would like to bring to the reader's attention an aspect that is often overlooked. More precisely, by 
following the same procedure as above, we would have
\begin{eqnarray}
	\bR_{k+1}&=&\bC-\mathcal{L}(X_{k+1}^l\bm{\tau_{k+1}}(X_{k+1}^r)^{\T}) \notag
\\
 &=&[C_1,A_1X_{k+1}^l,\ldots,A_\ell X_{k+1}^l]\begin{bmatrix}
                                 I&&&\\
                                 &-\bm{\tau}_{k+1}&&\\
                                 &&\ddots&\\
                                 &&&-\bm{\tau}_{k+1}\\
                                \end{bmatrix}
[C_2,B_1X_{k+1}^r,\ldots,B_\ell X_{k+1}^r]^{\T}.\label{eqn:res}
\end{eqnarray}
For a large number of terms $\ell$ in~\eqref{eqn:main}, explicitly allocating the matrices 
$[C_1,A_1X_{k+1}^l,\ldots,A_\ell X_{k+1}^l]$ and $[C_2,B_1X_{k+1}^r,\ldots,B_\ell X_{k+1}^r]$ may not be
affordable\footnote{Notice that in the multiterm Lyapunov case,
only the factor $[C,A_1X_{k+1},\ldots,A_\ell X_{k+1}]$ needs to be stored, by possibly 
rearranging the terms in the middle matrix containing the ${\bm \tau}_{i}$s;
see, e.g., \cite{SimonciniHao2023}.}.
This drawback is not a peculiarity of Algorithm~\ref{alg:subspaceCG},  as it often plagues solvers 
for~\eqref{eqn:main} as, e.g., the algorithms in~\cite{SimonciniHao2023},%
\cite{PalittaKuerschner2021},\cite{Biolietal2024}.

For $\ell$ larger than three or four, say,
we consider the use of two possible strategies to overcome this issue. 
The first one computes a number of thin QR factorizations, and
the second one relies on randomized numerical linear algebra tools.
In the following discussion, we employ an ad-hoc memory allocation 
threshold, namely {\tt maxrankR}, which
we set to be equal to $2\cdot${\tt maxrank} in our implementation.
Note that this value does not increase the actual requested 
memory allocations, as
the strategy employed for the iterates $X_{k+1}, P_{k+1}$ above requires 
storing two blocks of size {\tt maxrank} each; {see
Algorithm~\ref{alg:alphaCG}.}

\vskip 0.1in
{\it Dynamic truncated QR update}. In place of computing the thin 
QR of the whole matrices
$[C_1,\bA_1X_{k+1}^l\bm{\tau_{k+1}},\ldots,\bA_\ell X_{k+1}^l\bm{\tau_{k+1}}]$ and
$[C_2,-\bB_1X_{k+1}^r,\ldots,-\bB_\ell X_{k+1}^r]$ and then use their triangular
factors in the truncation, we sequentially combine $2\ell$ thin QR 
factorizations one after the other, detecting a possible linear dependency 
in the factors after each QR. More in detail,
we collect the subsequent products as follows:

\vskip 0.1in

\hskip 0.3in 
$Q^l {\bm r}^l=[C_1,\bA_1 X_{k+1}^l\bm{\tau_{k+1}}]$, \hskip 0.2in  $Q^r {\bm r}^r=[C_2, -\bB_1 X_{k+1}^r]$
\hskip 0.2in {\footnotesize{(QR factors of the two matrices)}}

\hskip 0.3in For $j=2, \ldots, \ell$

\hskip 0.5in $Q^l{\bm r}_1^l=[Q^l, \bA_j X_{k+1}^l\bm{\tau_{k+1}}]$, \hskip 0.2in
$Q^r {\bm r}_1^r=[Q^r, -\bB_j X_{k+1}^r]$ 
\hskip 0.2in {\footnotesize{(QR factors of the two matrices)}}

\hskip 0.5in ${\bm r}^l = {\bm r}_1^l \begin{bmatrix} {\bm r}^l & 0 \\ 0 & I \end{bmatrix}$
, \hskip 0.2in
${\bm r}^r = {\bm r}_1^r \begin{bmatrix} {\bm r}^r & 0 \\ 0 & I \end{bmatrix}.$

\vskip 0.1in

This procedure does not yet control the rank. To do so, 
the triangular matrices ${\bm r}_1^l$ and ${\bm r}_1^r$ are decomposed by means
of the SVD, so as to truncate the rank down to the maximum admittable
value {\tt maxrankR}. More precisely, if ${\bm r}_1^l=U\Sigma V^\T$ is the singular value
decomposition of ${\bm r}_1^l$, then the factors are truncated to the
most $i\le$  {\tt maxrankR} leading diagonal elements in $\Sigma$, so that
${\bm r}_1^l\approx U_{:,1:i}\Sigma_{1:i,1:i} V_{:,1:i}^\T$. Then,
the matrices $Q^l$ and ${\bm r}^l$ are updated accordingly\footnote{A rank revealing
QR decomposition could also be employed.}, that is
$$
Q^l = Q^l U_{:,1:i}, \qquad 
{\bm r}^l = \Sigma_{1:i,1:i} V_{:,1:i}^\T 
\begin{bmatrix} {\bm r}^l & 0 \\ 0 & I \end{bmatrix} .
$$
The same is done for $Q^r$ and ${\bm r}^r$.
%
At the end of the $j$-cycle, setting 
$R_{k+1}^l=Q^l$, $R_{k+1}^r=Q^r$, and $\vec{\rho}_{k+1}={\bm r}^l ({\bm r}^r)^\T$, 
we (re)define
$\bR_{k+1} := R_{k+1}^l \vec{\rho}_{k+1} (R_{k+1}^r)^\T$,
(not explicitly computed)
which is now an approximation to the true quantity in (\ref{eqn:res}).

\vskip 0.1in
{\it Randomized approximate QR update}.
The main goal is to compute tall matrices $Q$ and $W$ with orthonormal columns,
whose range is able to well represent the 
column- and row-space of $\bR_{k+1}$, respectively, i.e., $\text{range}(Q)\approx \text{range}(\bR_{k+1})$ and $\text{range}(W)\approx \text{range}(\bR_{k+1}^{\T})$. To this end, we apply the randomized range finder algorithm that can be found in, e.g.,~\cite[Algorithm 4.1]{Halko2010}. The first step in~\cite[Algorithm 4.1]{Halko2010} consists in multiplying the matrix at hand
($\bR_{k+1}$ and $\bR_{k+1}^{\T}$ in our case) by
a so-called \emph{sketching} matrix $G^l$ of conforming dimensions. The randomized nature 
of this sketching matrix is key for the success of the entire procedure. We will employ only Gaussian sketching matrices, though other options are available in the literature; see, e.g.,~\cite{Halko2010} for a discussion.

In our setting, given a target rank {\tt maxrankR}, we generate a Gaussian matrix $G^l\in\mathbb{R}^{n_B \times \mathtt {maxrankR}}$, and using (\ref{eqn:res})
we then compute
\begin{equation}\label{eqn:Romega}
	\bR_{k+1}G^l=
C_1(C_2^{\T}G^l)-\sum_{i=1}^\ell \bA_i(X_{k+1}^l\bm{\tau}_{k+1}((X_{k+1}^r)^{\T}(\bB_i G^l))).
\end{equation}
The algorithm proceeds with computing the $Q\in\mathbb{R}^{n_A\times \mathtt{maxrankR}}$ matrix of the thin
QR decomposition for the right-hand side matrix in (\ref{eqn:Romega}),
whose range is used as an approximation of the column-space of $\bR_{k+1}$.
The quality of this approximation strongly depends on the choice of the target 
rank $\mathtt{maxrankR}$ and on the decay rate of the singular values of
$\bR_{k+1}$; see, e.g.,~\cite[Theorem 9.1]{Halko2010}. 

The same procedure is adopted to compute 
$W\in\mathbb{R}^{n_B\times \mathtt{maxrankR}}$, with $\bR_{k+1}$
replaced by $\bR_{k+1}^{\T}$.

Once $Q$ and $W$ are computed, we use (\ref{eqn:res}) to perform
\begin{align*}
 Q^{\T}\bR_{k+1}W=
Q^{\T}C_1C_2^{\T}W-\sum_{i=1}^\ell (Q^{\T}\bA_iX_{k+1})
\bm{\tau}_{k+1}(X_{k+1}^{\T}\bB_iW)\in\mathbb{R}^{\mathtt{maxrankR}\times \mathtt{maxrankR}}.
\end{align*}
The procedure concludes by computing a truncated SVD of $Q^{\T}\bR_{k+1}W$, namely $Q^{\T}\bR_{k+1}W\approx \widehat U\widehat\Sigma\widehat V^{\T}$ providing the terms $R_{k+1}^l=Q\widehat U$, $\bm{\rho}_{k+1}=\widehat \Sigma$, and $R_{k+1}^r=W\widehat V$.

We would like to stress that the same sketching matrices $G^l$, $G^r$ 
can be used throughout the process, so that they can be generated  once for all
at the beginning of the iterative procedure. 

\rev{As already mentioned, in principle one could use non-Gaussian sketching matrices 
$G^l$ and $G^r$ and adopt, e.g., subsampled randomized trigonometric 
transformations (SRTT) for this task. 
However, we believe that employing Gaussian matrices is more appropriate  
in our context. Indeed, due to memory restrictions, we are able 
to allocate (at most) {\tt maxrankR} columns for the sketching matrices,
and Gaussian matrices provide better approximations 
than SRTTs for a fixed target rank, in general; 
see, e.g.,~\cite[Section 11.2]{Halko2010}. 
Achieving good rank-{\tt maxrankR} approximations to the residual 
matrix $\bR_k$ is crucial for the convergence of the overall \sscg\ method. 
Therefore, we always employ Gaussian matrices in spite of the slightly larger, 
yet negligible, cost in their application.
}

\vskip 0.1in

Numerical results reported in Table~\ref{ex_stoch:comparison} for Example~\ref{ex:stoch} 
for which $\ell=10$, show that the randomized update is superior to the dynamic
truncated procedure. Hence, in all other tests we either report results with
explicit computations of the products with 
$\bA_\star\bullet$, $\bB_\star\bullet$ or with the randomized strategy.
The corresponding procedure is summarized in
Algorithm~\ref{alg:Tres} in the Appendix, and it is named $\mathcal{T}_{res}$.

\subsection{Inaccurate coefficients}\label{Inaccurate coefficients} 
The computation of $\vec{\alpha}_k$ and  $\vec{\beta}_k$ requires the solution of an algebraic
equation with a linear operator having the same nature of the 
original $\cal L$, but with smaller dimension.
Up to a certain column dimension of $P_k$,
the matrices 
$\vec{\alpha}_k$, $\vec{\beta}_k$ can be computed by solving the related 
linear systems in Kronecker form by means of a direct solver.
Notice that 
the coefficient matrix of such linear systems has to be assembled 
only once to compute both $\vec{\alpha}_k$ and  $\vec{\beta}_k$. 
However, it may be the case,
either because of the not-so-small dimension, or because of the complexity of the operator $\cal L$,
the computation of $\vec{\alpha}_k$ and  $\vec{\beta}_k$ may have to be 
done via an iterative procedure such as the {truncated matrix-oriented \cg.}
If this is the case,
 these quantities are computed inexactly, that is, the exact matrices are
replaced by
approximate quantities. We denote them as
$\widetilde{\vec{\alpha}}_k= \vec{\alpha}_k + \vec{\epsilon}_k$, 
$\widetilde{\vec{\beta}}_k= \vec{\beta}_k + \vec{\eta}_k$.
To ease the connection with the derivations 
of section~\ref{Truncated matrix-oriented CG}  and section~\ref{sec:discussion},
in the following we use the symmetric notation for the factors.
The inexact solutions $\widetilde{\vec{\alpha}}_k$ and $\widetilde{\vec{\beta}}_k$
no longer grant the orthogonality properties associated with the exact 
quantities $\vec{\alpha}_k$ and  $\vec{\beta}_k$.
In particular, by defining $ \widetilde \bR_{k+1}  = \bR_k - {\cal L}(P_k \widetilde{\vec{\alpha}}_k P_k^\T)$,
the orthogonality property $P_k^\T \widetilde \bR_{k+1} P_k = 0$ is lost.
Nonetheless, loss of local orthogonality can be 
tracked at each iteration, and directly related to the accuracy
with which the small problems are solved. This is described in the following proposition.

\begin{proposition}
Let $\widetilde{\vec{\alpha}}_k$ be the approximate solution to
$P_k^\T {\cal L}(P_k \vec{\alpha}_k P_k^\T)P_k = P_k^\T \bR_k P_k$, and
let $\vec{\varrho}_k$ be the associated residual matrix.  Then
$$
P_k^\T \widetilde \bR_{k+1} P_k = \vec{\varrho}_k,
$$
where it also holds that $P_k^\T \widetilde \bR_{k+1} P_k = P_k^\T (\widetilde \bR_{k+1}-
\bR_{k+1}) P_k$.
\end{proposition}
\begin{proof}
We notice that $\vec{\varrho}_k = P_k^\T {\cal L}(P_k {\vec{\epsilon}}_k P_k^\T)P_k$.
The residual recurrence gives
$\widetilde \bR_{k+1}  = \bR_{k}  -  {\cal L}(P_k \widetilde{\vec{\alpha}}_k P_k^\T) =
	\bR_{k+1}  -  {\cal L}(P_k {\vec{\epsilon}}_k P_k^\T)$.
Hence,
$P_k^\T \tilde \bR_{k+1} P_k =
	P_k^\T \bR_{k+1}P_k   -  P_k^\T {\cal L}(P_k{\vec{\epsilon}}_k P_k^\T) P_k =
\vec{\varrho}_k$.
\end{proof}

A similar relation holds for the residual matrix associated with the
reduced matrix equation yielding the coefficient $\widetilde{\vec{\beta}}_k$.

Inexactness also implies loss of
orthogonality with respect to the previous iterates.
However, in a context where all iteration
 factors are anyway truncated, we do not expect this truncation to have
particularly strong implications.
In most of our numerical experiments we compute 
$\vec{\alpha}_k$ and $\vec{\beta}_k$ \emph{exactly}, by solving the 
related linear systems by a direct solver. 
Indeed, thanks to the rather moderate caps on the rank of the iterates we 
adopt, this operation does not remarkably affect the computational cost 
of the overall iterative solver.

\section{Preconditioning}\label{sec:precond}
As the number of iterations increases, the iterates rank grows,
possibly compelling a systematic use of truncation.
Since information may be lost during truncation, strategies
to accelerate convergence so as to decrease  the number of iterations
need to be devised. As in standard \cg, preconditioning
the coefficient operator is a natural strategy. 
Equipping Algorithm~\ref{alg:subspaceCG} with a preconditioner 
does not follow the same exact lines as what is done for matrix-oriented  \cg.
Indeed, the different computation 
of $\vec{\alpha}_k$ and $\vec{\beta}_k$ leads to a different 
handling of the preconditioned quantities. We derive the transformed
recurrence closely following the procedure in \cite[Section 11.5.2]{Golub.VLoan.13} 
employed for vector \cg.

Consider a preconditioning operator 
$\mathcal{P}:\mathbb{R}^{n_A\times n_B}\rightarrow\mathbb{R}^{n_A\times n_B}$ defined by
symmetric matrices, and positive definite with respect to the matrix inner product. 
We restrict our attention to invertible operators such that
\begin{eqnarray}\label{eqn:prec}
	{\cal P}^{-1}( Y^l (Y^r)^\T) = [\bH_1^{l} Y^l, \ldots, \bH_t^{l} Y^l] \vec{\kappa}
	[\bH_1^{r} Y^r, \ldots, \bH_t^{r} Y^r]^\T =  [\bH_\star^{l} \bullet Y^l] \vec{\kappa} [\bH_\star^{r} \bullet Y^r]^\T,
\end{eqnarray}
for some matrix $\vec{\kappa}$ of conforming dimensions.
Under these hypotheses, there exists a nonsingular symmetric
operator $\mathcal{G}:\mathbb{R}^{n_A\times n_B}\rightarrow\mathbb{R}^{n_A\times n_B}$ such
that\footnote{In Kronecker form, 
${\cal G}$ corresponds to the square root of the positive definite Kronecker form
of ${\cal P}$.}
$\mathcal{P}(\bX)=\mathcal{G}(\mathcal{G}(\bX))$.
In place of $\mathcal{L}(\bX)=\bC$, we can thus solve the equivalent equation
$$\mathcal{G}^{-1}(\mathcal{L}(\mathcal{G}^{-1}(\mathcal{G}(\bX))))=\mathcal{G}^{-1}(\bC).$$
By setting $\widetilde{\mathcal{L}}=
\mathcal{G}^{-1}\mathcal{L}\mathcal{G}^{-1}:\mathbb{R}^{n_A\times n_B}\rightarrow\mathbb{R}^{n_A\times n_B}$,
$\widetilde{\bX}=\mathcal{G}(\bX)$, 
and $\widetilde{\bC}=\mathcal{G}^{-1}(\bC)$, we now apply 
the new scheme to the equation
$\widetilde{\mathcal{L}}(\widetilde{\bX})=\widetilde{\bC}$.

By starting with $\widetilde{\bR}_0=\widetilde{\bC}-\widetilde{\mathcal{L}}(\widetilde{\bX}_0)=\mathcal{G}^{-1}(\bC-\mathcal{L}(\bX_0))$ for a given $\widetilde{\bX}_0$, and setting $\widetilde{\bP}_0=\widetilde{\bR}_0$, the \sscg\ iteration becomes
$$
\widetilde{\bX}_{k+1}=
\widetilde{\bX}_{k}+\widetilde P_k^l\widetilde{\vec{\alpha}}_k (\widetilde P_k^r)^{\T}, \quad
\widetilde{\bR}_{k+1}=
\widetilde{\bC}-\widetilde{\mathcal L}(\widetilde{\bX}_{k+1}), \quad
\widetilde{\bP}_{k+1}=
\widetilde{\bR}_{k+1}+\widetilde P_k^l\widetilde{\vec{\beta}}_k(\widetilde P_k^r)^{\T}.
$$
This can be rewritten as
$
\bX_{k+1}=\bX_{k}+\mathcal{G}^{-1}(\mathcal{G}^{-1}( P_k^l\vec{\alpha}_k (P_k^r)^{\T}))=\bX_{k}+\mathcal{P}^{-1}( P_k^l\vec{\alpha}_k (P_k^r)^{\T})=\bX_{k}+ \widehat P_k^l\widehat{\vec{\alpha}}_k (\widehat P_k^r)^{\T},
$
where $\widehat P_k^l\widehat{\vec{\alpha}}_k (\widehat P_k^r)^{\T}$ is the (possibly low rank)
factorization of the result of applying ${\cal P}^{-1}$.
Moreover, $\bR_{k+1}=\bC-\mathcal L(\bX_{k+1})$ and
\begin{align*}
\mathcal{G}^{-1}(\bP_{k+1})=&\,\mathcal{G}^{-1}(\bR_{k+1})+\mathcal{G}^{-1}( P_k^l\vec{\beta}_k(P_k^r)^{\T}).
\end{align*}
By applying $\mathcal{G}^{-1}$ to the last relation we get $\mathcal{P}^{-1}(\bP_{k+1})=\mathcal{P}^{-1}(\bR_{k+1})+\mathcal{P}^{-1}( P_k^l\vec{\beta}_k(P_k^r)^{\T})$, namely
$$
\widehat{\bP}_{k+1}=\bZ_{k+1}+\widehat P_k^l\widehat{\vec{\beta}}_k(\widehat P_k^r)^{\T},
\quad \text{where }\bZ_{k+1}:=\mathcal{P}^{-1}(\bR_{k+1}).
$$

\vskip 0.05in
\begin{remark}
{\rm 
The computation  $\mathcal{P}^{-1}(\bP_{k+1})=
\mathcal{P}^{-1}(\bR_{k+1})+\mathcal{P}^{-1}( P_k^l\vec{\beta}_k(P_k^r)^{\T})$
above  should in fact be understood as that the range of
$\mathcal{P}^{-1}(\bP_{k+1})$ equals the range of
$\mathcal{P}^{-1}(\bR_{k+1}+ P_k^l\vec{\beta}_k(P_k^r)^{\T})$. In particular, by using
(\ref{eqn:prec}), the
range of $\mathcal{P}^{-1}(\bP_{k+1})$ is obtained by using the
range of $[\bH_\star^l\bullet R_{k+1}, \bH_\star^l\bullet P_k^l]$.
The same holds for the transpose of $\mathcal{P}^{-1}(\bP_{k+1})$.
$\hfill\square$
}
\end{remark}
\vskip 0.05in

To sum up, the preconditioned variant of the subspace conjugate-gradient method (\sscg)
starts by setting $\bZ_0:=\mathcal{P}^{-1}(\bR_0)$ and $\bP_0=\bZ_0$ and then at
the $k$th iteration it proceeds as (in the final implementation each of these
assignments will undergo a proper truncation step for generating the low-rank factors)
\begin{align*}
\bX_{k+1}=&\,\bX_{k}+P_k^l\vec{\alpha}_k (P_k^r)^{\T}, \quad
\bR_{k+1}=\bC-\mathcal L(\bX_{k+1})\\
\bZ_{k+1}=&\,\mathcal{P}^{-1}(\bR_{k+1}),\quad
\bP_{k+1}=\bZ_{k+1}+ P_k^l\vec{\beta}_k (P_k^r)^{\T}.
\end{align*}
The matrix $\vec{\alpha}_k$ is again the minimizer of 
$\phi(\vec{\alpha})$ in~\eqref{eq:min:alpha}, with the newly
``preconditioned'' directions $P_k^l, P_k^r$.
To maintain the $\mathcal{L}$-orthogonality of the directions $\bP_i$'s, 
we compute $\vec{\beta}_k$ as the solution to the projected equation
$$
(P_k^l)^{\T}\mathcal{L}(P_k^l\vec{\beta}_k(P_k^r)^{\T})P_k^r=
-(P_k^l)^{\T}\mathcal{L}(\bZ_{k+1})P_k^r.
$$
We are left to select the actual preconditioning operator to perform
$\bZ_{k+1}=\mathcal{P}^{-1}(\bR_{k+1})=
\mathcal{P}^{-1}(R_{k+1}^l\bm{\rho}_{k+1}(R_{k+1}^r)^{\T})$.
 Most preconditioning operators in the relevant literature are of the form
$$\mathcal{P}_1(\bX)=\bE\bX\bD, \quad\text{or}\quad \mathcal{P}_2(\bX)=\bE\bX\bD+\bF\bX\bG.$$
Indeed, applying $\mathcal{P}^{-1}$ 
turns out to be affordable only when $\mathcal{P}$ is itself a linear operator with 
at most two terms; see, 
e.g.,~\cite{Biolietal2024},\cite{Voet2023},\cite{Powell2009},\cite{Ullmann2010},\cite{Stoll.Breiten.15}.
The operation $\mathcal{P}_1^{-1}$ corresponds to inverting $\bE$ and $\bD$, 
namely $\mathcal{P}^{-1}_1(\bX)=\bE^{-1}\bX \bD^{-1}$, thus
perfectly matching the condition (\ref{eqn:prec}). On the other hand,
to comply with (\ref{eqn:prec}), the application of 
${\cal P}_2^{-1}$ can
be performed by using a method like low-rank ADI (LR-ADI)
 for Sylvester equations (see \cite{ADISylv}),
whose approximate solution can be shown to satisfy this expression;
see \cite[Proposition 3.1]{Druskin.Knizhnerman.Simoncini.11}.
For completeness we mention that to the best of our knowledge, 
the only option where $\mathcal{P}$ has more than two terms is proposed 
in~\cite[Section 4]{Voet2023} where $\mathcal{P}^{-1}$ is computed as an 
approximation to $\mathcal{L}^{-1}$ in Kronecker form with $q$ (possibly sparse) terms. 

In all our experiments, {whenever $\mathcal{P}_2$ is employed we apply it by running} $t_{ADI}$ iterations of
LR-ADI for Sylvester equations,
where we use $t_{ADI}$ (sub)optimal Wachspress’ shifts~\cite{ADIshifts}.
 According to (\ref{eqn:prec}), the resulting
left and right factors will each  have $t_{ADI}\cdot r_{k+1}$ columns whenever
the input factor $R_{k+1}^l$ has $r_{k+1}$ columns. We stress that the algorithm
in \cite{ADIshifts} uses the same shifts for the left and right
sequences. \rev{Our implementation of LR-ADI is the same as the one 
proposed in \cite{Biolietal2024}, except for the matrix sparsification\footnote{\rev{The code made
available by the authors in the repository cited in \cite{Biolietal2024} used matrices in full format, instead
of sparse format}.}.}
We have not \rev{further} modified the code, since we used ${\cal P}_2$ in a context
when the left and right shifts could be the same.

%
To limit memory allocations, and according to what was already
described for the application of the operator $\cal L$,
we perform a {\it dynamic} 
low-rank truncation of the current iterate factors 
{\it at each} LR-ADI iteration. If this truncation is based on
a {\tt maxrank} policy (cf. 
section~\ref{sec:truncation}), this implementation of LR-ADI
allocates at most $4\cdot\mathtt{maxrank}$ columns
(half of which for either the right or left factor)
regardless of the number of iterations.  Since this procedure
can significantly worsen the preconditioner quality, it should
only be adopted under
severe storage constraints.

\section{The complete algorithm}\label{sec:final_algo}

The new preconditioned subspace-conjugate gradient method (\sscg)
for multiterm Sylvester equations 
is summarized in Algorithm~\ref{alg:alphaCG},
equipped with the computational 
advances described in the previous sections.

Special attention deserves the stopping criterion.
While the randomized procedure is able to remarkably
reduce the memory requirements of the overall solution process, 
it does not construct an approximation of 
the norm  of the residual matrix $\bC-\mathcal{L}(\bX_{k+1})$ that can be used 
to assess the accuracy of $\bX_{k+1}$.
As an alternative common choice (see, e.g., \cite{Powelletal2017}), we 
monitor the difference between two consecutive approximate solutions as stopping criterion:
%
\begin{equation}\label{eqn:stopping}
\|\bX_{k+1}-\bX_k\|_F/\|\bX_{k+1}\|_F\leq\mathtt{tol}.
\end{equation}
The computation of the Frobenius norms exploits
the fact that the columns $X_j^l$ and $X_j^r$ are kept orthonormal for every $j$, so that
$\|X_{k+1}^l\bm{\tau}_{k+1}(X_{k+1}^r)^{\T}\|_F=\|\bm{\tau}_{k+1}\|_F,$
and
$$
\|X_{k+1}^l\bm{\tau}_{k+1}(X_{k+1}^r)^{\T}-X_{k}^l\bm{\tau}_{k}(X_{k}^r)^{\T}  \|_F^2
=\|\bm{\tau}_{k+1}\|_F^2+\|\bm{\tau}_{k}\|_F^2-2
\text{trace}
(\bm{\tau}_{k+1}(X_{k+1}^r)^{\T}X_{k}^l\bm{\tau}_{k}(X_{k}^r)^{\T}X_{k+1}^l).
$$

\vskip 0.05in
\rev{
\begin{remark}
It is known that when used with iterative methods, the stopping criterion in (\ref{eqn:stopping}) may be 
sensitive to the ill-conditioning of the problem, as a small relative difference
does not necessarily correspond to a small residual. In our setting, however, we recall that 
stagnation of the approximate solution is more likely to occur as an intrinsic effect of 
the truncation step.
The criterion (\ref{eqn:stopping}) was chosen because computing 
the true residual norm is expensive, as previously discussed.
Nonetheless, if one is interested in monitoring the residual,
a careful implementation may consider including an estimation of the true residual norm
once the criterion (\ref{eqn:stopping}) is satisfied. In case such estimate is not satisfactorily small,
the iteration will continue a few more steps, until the residual norm either
converges or stagnates, or the
maximum number of allowed iterations is reached. The estimation of the true residual 2-norm
can be obtained -- not inexpensively -- by running a few iterations of a sparse SVD iterative method
using the factorized version of ${\bm R}_{k+1}$.
This variant is not included in the experiments below. However, the true residual norm was
computed at completion for all considered methods.
\end{remark}
}
\vskip 0.05in

Concerning the preconditioning step, when 
writing $\mathcal{P}^{-1}(R_{k+1}^l\bm{\rho}_{k+1}(R_{k+1}^r)^\T)$ in
lines~\ref{alg:Zline1} and~\ref{alg:Zline2} we mean that the preconditioner 
is applied as in \eqref{eqn:prec}, without explicitly assembling 
$R_{k+1}^l\bm{\rho}_{k+1}(R_{k+1}^r)^\T$.
The result of this operation can be further truncated by the QR-SVD operator 
$\mathcal{T}$ in~\eqref{eq:QRSVD}.
The shorthand notations {\tt params} and {\tt params\_res} for $\mathcal{T}$ and $\mathcal{T}_{res}$ resp., indicate the inclusion of all the necessary input parameters. 

Finally, concerning  memory requirements, 
in addition to all iterates' factors -- each requiring (at most) 
$\mathtt{maxrank}(n_A+n_B)+\mathtt{maxrank}^2$ allocations -- the method uses
working storage, reported in Algorithm~\ref{alg:alphaCG}. In particular, the storage allocation needed by $\mathcal{T}_{res}$ (lines~\ref{alg:lineTres1} and~\ref{alg:lineTres2}) depends on the adopted procedure: $(\ell\mathtt{maxrank}+s_C)(n_A+n_B)$ for the QR-SVD truncation or $\mathtt{maxrankR}(n_A+n_B)$ for the randomized strategy illustrated in section~\ref{sec:truncation}. 
{Similarly, in line~\ref{alg:Zline1} and line~\ref{alg:Zline2} 
memory requirements for $\bZ_k$ correspond to $\mathtt{maxrankR}(n_A+n_B)$ for
${\cal P}_1$, and they will be higher when using ${\cal P}_2$.}
Memory requirements are
comparable with those of {\sc tcg}, except for $s_k^4$ in line~\ref{alg:linesk4} {that is due to the allocation of the coefficient matrix of the Kronecker form of~\eqref{eq:prop:alpha:thesis}}.

\begin{algorithm}[bt]
{\footnotesize
\begin{algorithmic}[1]
\smallskip
\Statex \textbf{Input:} Operator
$\mathcal L:\mathbb{R}^{n_A\times n_B}\rightarrow \mathbb{R}^{n_A\times n_B}$,
preconditioner
$\mathcal P:\mathbb{R}^{n_A\times n_B}\rightarrow \mathbb{R}^{n_A\times n_B}$,
right-hand side factors $C_1$, $C_2$ $(\bC=C_1C_2^{\T})$, 
initial guess factors $X_0^l$, $X_0^r$, $\bm{\tau}_0$ $(\bX_0=X_0^l\bm{\tau}_0(X_0^r)^{\T})$, 
max no. iterations $\texttt{maxit}$, 
tolerance $\texttt{tol}$, low-rank truncation tolerances {\tt tolrank, maxrank, maxrankR}, flag {\tt flag\_rsvd}.
\Statex \textbf{Output:} Approximate solution 
factors  $X_k^l$, $X_k^r$, $\bm{\tau}_k$ $(\bX_k=X_k^l\bm{\tau}_k(X_k^r)^{\T})$  such that $\|\bX_k-\bX_{k-1}\|\leq \|\bX_k\| \cdot \texttt{tol}$
\smallskip

\If{{\tt flag\_rsvd}}
\State Create random Gaussians $G^l\in\mathbb{R}^{n_B\times \mathtt{maxrankR}}$,
$G^r\in\mathbb{R}^{n_A\times \mathtt{maxrankR}}$\label{alg:defineOmega}
\textbf{else}
Set $G^l=G^r=\emptyset$
 \EndIf
\State Set $[R_0^l,\bm{\rho}_0,R_0^r]=\mathcal{T}_{res}(C_1, C_2,X_0^l,\bm{\tau}_0,X_0^r,\mathtt{params\_res})$\Comment{{\footnotesize $(\ell\mathtt{maxrank}+s_C)(n_A+n_B)$ or $\mathtt{maxrankR}(n_A+n_B)$}}\label{alg:lineTres1}
\State Compute $[Z_0^l,\bm{\zeta}_0,Z_0^r]=\mathcal{T}(\mathcal{P}^{-1}(R_0^l\bm{\rho}_0(R_0^r)^{\T}),\mathtt{params})$ \Comment{{\footnotesize$\mathtt{maxrank}(n_A+n_B)$ (at least)}}\label{alg:Zline1}
 \State Set $P_0^l=Z_0^l$, $P_0^r=Z_0^r$ \For{$k=0,\ldots,\mathtt{maxit}$}
\State Compute $\vec{\alpha}_k$ by solving \Comment{{\footnotesize$s_k^4$}}
$$
(P_k^l)^\T {\cal L}(P_k^l \vec{\alpha}_k (P_k^r)^\T)P_k^r = (P_k^l)^\T R_k^l\bm{\rho}_k (R_k^r)^{\T} P_k^r
$$ \label{alg:linesk4}
\State Set $[X_{k+1}^l,\bm{\tau}_{k+1},X_{k+1}^r]=\mathcal{T}([X_k^l,P_k^l],\text{blkdiag}(\bm{\tau}_k,\bm{\alpha}_k),[X_k^r,P_k^r],\mathtt{params})$ \Comment{{\footnotesize $2(n_A+n_B)\mathtt{maxrank}$}} \If{(\ref{eqn:stopping}) holds }
\State Return $X_{k+1}^l,\bm{\tau}_{k+1},X_{k+1}^r$
\EndIf
\State Set $[R_{k+1}^l,\bm{\rho}_{k+1},R_{k+1}^r]=\mathcal{T}_{res}(C_1, C_2,X_{k+1}^l,\bm{\tau}_{k+1},X_{k+1}^r, \mathtt{params\_res})$ \Comment{{\footnotesize $(\ell\mathtt{maxrank}+s_C)(n_A+n_B)$ or }}\label{alg:lineTres2}
\Statex \Comment{{\footnotesize $\mathtt{maxrankR}(n_A+n_B)$}}

\State Compute $[Z_{k+1}^l,\bm{\zeta}_{k+1},Z_{k+1}^r]=\mathcal{T}(\mathcal{P}^{-1}(R_{k+1}^l\bm{\rho}_{k+1}(R_{k+1}^r)^{\T}),\mathtt{params})$\Comment{{\footnotesize$\mathtt{maxrank}(n_A+n_B)$ (at least)}}\label{alg:Zline2}
\State Compute $\vec{\beta}_k$ by solving 
$$
(P_k^l)^\T {\cal L}(P_k^l \vec{\beta}_k (P_k^r)^\T)P_k^r = (P_k^l)^\T Z_{k+1}^l\bm{\zeta}_{k+1} (Z_{k+1}^r)^{\T} P_k^r
$$
\State Set $[P_{k+1}^l,\bm{\gamma}_{k+1},P_{k+1}^r]=\mathcal{T}([Z_{k+1}^l,P_k^l],\text{blkdiag}(\bm{\zeta}_{k+1},\bm{\beta}_k),[Z_{k+1}^r,P_k^r],\mathtt{params})$    \Comment{{\footnotesize$2(n_A+n_B)\mathtt{maxrank}$}}

\EndFor
\end{algorithmic}
}
\caption{Preconditioned subspace-conjugate gradient method (\sscg)}\label{alg:alphaCG}
\end{algorithm}

\section{Numerical results}\label{Numerical results}
This section serves as introduction to the upcoming numerical experiments.
All data are either publicly available or can be easily constructed.
Our computational analysis has two goals: first, we illustrate the properties
of the new method. We explore how the choice of the maximum rank
influences the performance. To this end, in all results we always set {\tt tolrank}=$10^{-12}$. Moreover, we report the convergence behavior
when the different strategies of section~\ref{sec:truncation} are adopted to
deal with memory requirements associated to the construction of the residual matrix.

Second, we compare the performance of our new method with that of state-of-the-art
algorithms for the same problems. More precisely:

\vskip  0.05in

{\tpcg}: truncated preconditioned \cg, as recalled in\footnote{A possible implementation 
is available at {\tt http://www.dm.unibo.it/\textasciitilde simoncin/tcg.tar.gz}}
section~\ref{Truncated matrix-oriented CG}. The maximum allocated
memory corresponds to twice {\tt maxrank} long vectors for each recurrence, 
while for the residual computation {we need to allocate $\ell\mathtt{maxrank}+s_C$ long vectors}.
\vskip  0.05in

{\sc sss}: Fixed-point iteration method for multiterm Lyapunov equations, with inexact 
solves with the leading portion of the operator, given by the first two terms%
\footnote{The Matlab code is publicly available at 
{\tt https://www.dm.unibo.it/\textasciitilde simoncin/software.html}.}
\rev{\cite{Shanketal.16}}.
 The
allocated memory cannot be predicted a priori, as it depends on the memory required
by the inner iterative projection solver. In our experiments this is reported a-posteriori, and
it is generally unrelated to {\tt maxrank}. Moreover, the final rank of the approximate
solution is not fixed a-priori, and it is the result of a truncation
to the small threshold $10^{-14}$, as suggested in the original code.
\vskip  0.05in

{\sc r-nlcg}:
Optimization approach approximating the solution on manifolds of maximum-rank matrices through
Riemannian Conjugate Gradient, with incorporation of preconditioners \cite{Biolietal2024}.
It will be used for multiterm Sylvester equations\footnote{The Matlab code is publicly available at {\tt https://github.com/IvanBioli/riemannian-spdmatrixeq}}.
 The algorithm terminates if the gradient norm drops below {\tt tolgradnorm }$ = 10^{-6}\Vert\mat{C}\Vert$ 
or if the norm of the displacement vector (to be retracted) is smaller than {\tt minstepsize }$ = 10^{-6}$.
Memory allocations are not declared in the article. However
scrutiny of the code seems to show that the memory employed is actually significant.
For instance, the residual factors $\bA_\star\bullet X_1$
and $\bB_\star\bullet X_2$ are computed explicitly.
\vskip  0.05in

{MultiRB}: Projection method specifically designed for finite element discretizations
of differential equations with stochastic inputs\footnote{\rev{The Matlab code is publicly available at {\tt https://www.dm.unibo.it/\textasciitilde simoncin/software.html}}} \rev{\cite{Powelletal2017}}. 
The method enforces
a Galerkin condition for the spatial variables, with respect to
a rational Krylov subspace specifically tailored to the problem,
while the random space is not reduced.  As the stopping criterion is based on
tolerance, memory allocations and final rank can only be monitored a posteriori.

\vskip  0.05in
Unless explicitly stated,
for \sscg\ we consider both truncation variants of the
residual matrix, that is the deterministic version with the factor fully allocated
($\mathcal{T}_{res}$ with $G^l=G^r=\emptyset$, labeled \sscg\ \emph{determ})
and the randomization-based one ($\mathcal{T}_{res}$
with $G^l=G^r\neq\emptyset$, labeled \sscg\ \emph{rand'zed}).
In all instances, the solution of the matrix equations to determine ${\bm \alpha}_k$
and ${\bm \beta}_k$ was carried out with the problem in Kronecker form up to dimension
$4000$ of the Kronecker matrix.
{For all CG-type methods and MultiRB, the stopping criterion 
was based on (\ref{eqn:stopping})
while a cheap bound is used for {\sc sss}; except for Example~\ref{ex:param},
the stopping  tolerance was set to $10^{-6}$. 
Algorithm {\sc r-nlcg} used multiple stopping criteria, with values set above.}
\rev{This parameter tuning ensured that all the different solvers attain similar solutions, in general. In particular, the final true residual norm was computed at completion (but excluded
in the total costs), to double check that all solutions have comparable accuracy.}
{The running time is marked whenever
the residual norm was smaller (over-solving) or larger (under-solving) by at least
one order of magnitude with respect to those of the other  methods.}

{All the experiments have been run using Matlab (version 2024b) on a machine
with a 4.4GHz Intel 10-core CPU, including two high-performance cores and eight high-efficiency cores,
 equipped with i5 processor with 16GB RAM on an Ubuntu 2020.04.2
LTS operating system.}
\subsection{Numerical experiments for the
multiterm Lyapunov equation}\label{sec:expes_lyap}
In this section we report a selection of results from our computational experience with Algorithm~\ref{alg:alphaCG} applied to the multiterm Lyapunov equation
\begin{equation}\label{eqn:Lsym}
\bA \bX + \bX \bA + \bM \bX \bM = \bC.
\end{equation}
%
For this problem, preconditioning is naturally two-term (cf. $\mathcal{P}_2$ in section~\ref{sec:precond}), especially if the
operator ${\cal L}_0\,:\, \bX \rightarrow \bA \bX + \bX \bA$ is, in some sense, the dominant part of
the whole $\mathcal{L}$. Both examples below thus use ${\cal L}_0$ as
preconditioner, and the action of its inverse is approximated as described in section~\ref{sec:precond}.

%

\begin{example}\label{ex:diffreac}
{\rm
This first example aims at illustrating the convergence behavior of the new method
with respect to the chosen values   of  the parameters {\tt maxrank} and {\tt tol}.
We  thus focus on number of iterations as  quality measure,
postponing to subsequent experiments the use of running time.
We consider the discretization by centered finite differences of 
the partial differential equation
$$
(\theta(x) u_x)_x + (\theta(y) u_y)_y
 + \gamma(x,y) u = f,\,\, \mbox{ with}\,\, (x,y)\in (0,1)^2,
$$
and Dirichlet zero boundary conditions. 
Here $f$ is constant and equal to one in the whole domain, while
$\theta(z) = -\frac 1 {10} \exp(-z)$.
Each factor in the second order term can be discretized by a
three-point stencil that deals with one-dimensional
second order derivatives with nonconstant coefficients. By doing so, we obtain a matrix $\bA$ which is symmetric, since we are working on the unit square discretized with a uniform mesh, and tridiagonal having components
$$
\bA = 
\frac 1 {h^2}{\rm tridiag}(A_{i,i-1}, {A_{i,i}}, A_{i,i+1}), \qquad 
A_{i,i\pm 1}=\theta(x_{i\pm \frac 1 2}), \, 
A_{i,i}=-(\theta(x_{i- \frac 1 2})+\theta(x_{i+\frac 1 2})), 
$$
where the $x_j$s are the discretization nodes in each direction, and $x_{j\pm \frac 1 2}$
are values at the midpoint of each discretization subinterval.
	The reactive coefficient $\gamma(x,y)$ is separable,
that is $\gamma(x,y) = \gamma_0(x)\gamma_0(y)$, for
two different settings:
\vskip 0.1in
i) $\gamma_0(z) =  \sin(z\pi)$, with $z\in (0,1)$;
\hskip 0.4in ii) $\gamma_0(z) =  \exp(z\pi)$, with $z\in (0,1)$.
\vskip 0.1in
At the discrete level, the reactive term can thus be represented 
by $\bM\bX\bM$ with $\bM$ diagonal and having on its diagonal the 
nodal values of $\gamma_0$.
The discretization  employs $n_A=8000$ interior nodes in each direction
so that $\bA,\bM\in\mathbb{R}^{n_A\times n_A}$.

\begin{figure}[t]
\centering
\includegraphics[scale=0.3,height=1.9in]{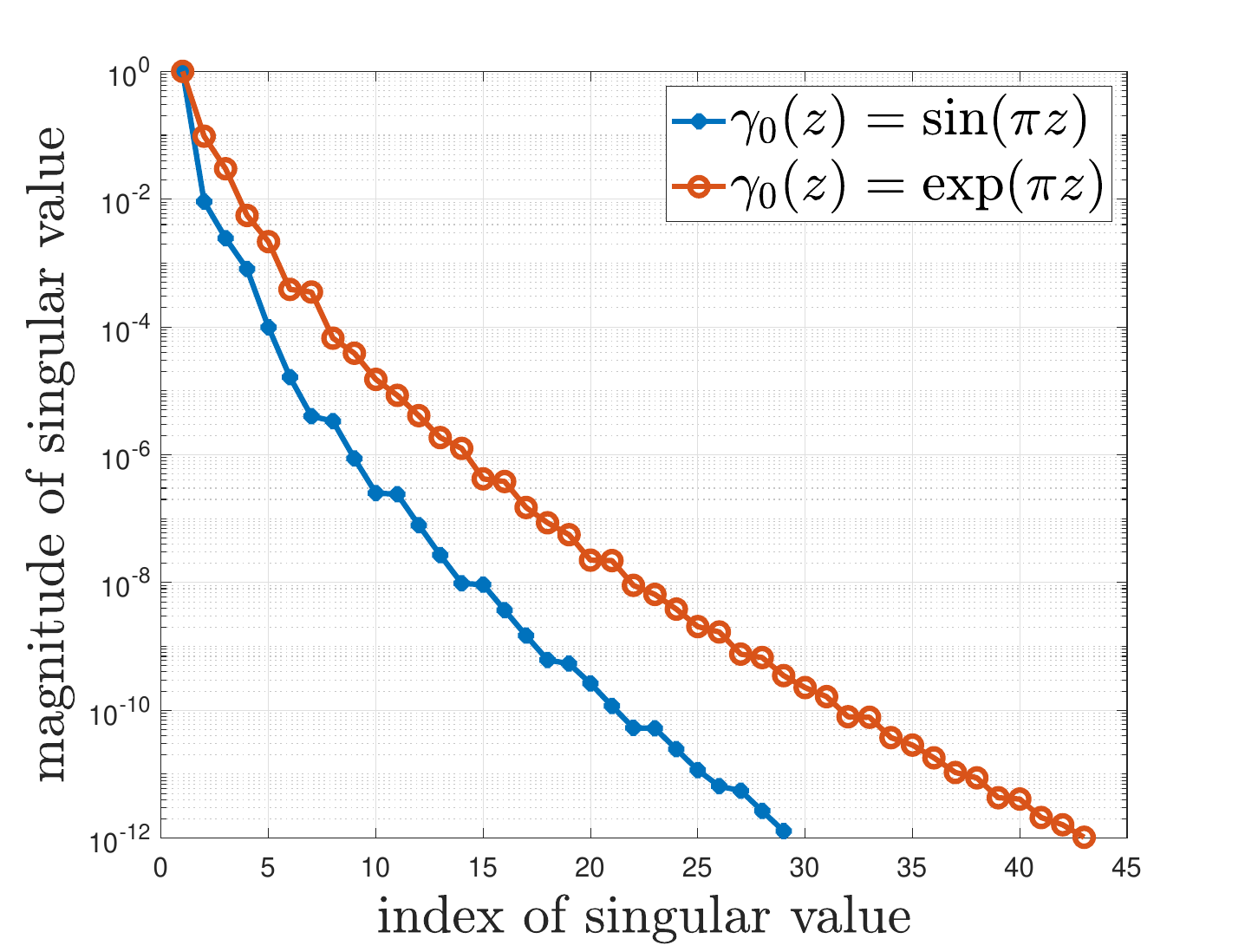}
\includegraphics[scale=0.46,height=1.9in]{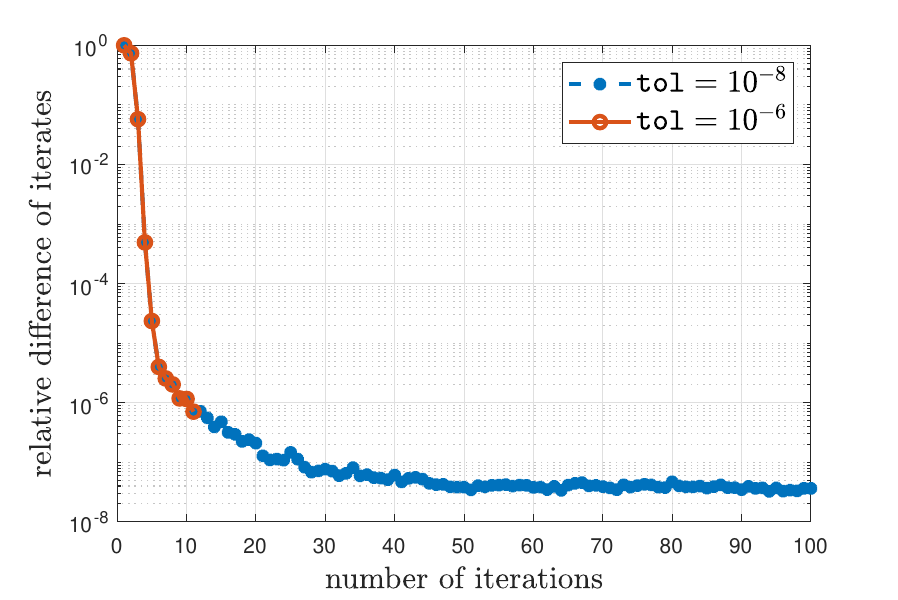}
\caption{Example \ref{ex:diffreac}. Left: Singular value distribution
of the approximate solution matrix $\bX$ for $\gamma_0(z)=\sin(z\pi)$ 
and $\gamma_0(z)=\exp(z\pi)$. Right:
Convergence history of \sscg\ for $\gamma_0(z)=\exp(z\pi)$, different stopping tolerances {\tt tol}, and
fixed {\tt maxrank}=20.  \label{fig:svdX}}
\end{figure}

The left plot In Figure~\ref{fig:svdX} reports the leading singular value distribution of
the solution $\bX$  (obtained with higher accuracy than in the following experiments)
for $\gamma_0(z) =  \sin(z\pi)$ and $\gamma_0(z) =  \exp(z\pi)$.
The different singular value decay for the two cases is clearly visible,
providing insight into what to expect when running truncation-based methods:
the slower decay for $\gamma_0(z)=\exp(z\pi)$ suggests 
that the method will require higher rank iterates to be able to compute 
an accurate solution in this case. In other words, the stopping tolerance and
the maximum rank cannot be selected disjointly.


Table~\ref{tab:pde_analysis} shows the performance of \sscg\ in terms of
number of iterations, for different choices
of {\tt maxrank} and the final {\tt tol}. 
The two-term preconditioner ${\cal P}_2\equiv\mathcal{L}_0$ was used, running
$t_{ADI}=8$ LR-ADI iterations. For this test case, no randomization is adopted,
so as to analyze the fully deterministic setting. Truncation is driven by the 
{\tt maxrank} parameter.
%
For the reactive coefficient $\gamma_0(z) =  \sin(z\pi)$, the fast singular value
decay ensures convergence in few iterations for both tested tolerances for the chosen
 very small maximum rank.
As suggested by the discussion above on the decay of the singular values 
of $\bX$, the scenario changes for $\gamma_0(z)=\exp(z\pi)$: 
for {\tt maxrank}$=20$, \sscg\ convergences rather fast if
{\tt tol}$=10^{-6}$, but it
is not able to meet the prescribed accuracy in 100 iterations if {\tt tol}$=10^{-8}$.
The right plot of 
Figure~\ref{fig:svdX} reports the convergence detail for {\tt maxrank}=20 and the two different
values of {\tt tol}: for {\tt tol}$=10^{-8}$ stagnation can be observed around the 
threshold, as this choice of {\tt maxrank}
is too small to capture the first {\tt maxrank} singular values of $\bX$ sufficiently well.
As can be seen in Table~\ref{tab:pde_analysis}, this issue gets fixed by increasing the value 
of {\tt maxrank} with a remarkable reduction in the iteration count by increasing
the maximum rank further.
This pattern is typical in all our subsequent experiments, giving as feedback that if
stagnation occurs, the value of
{\tt maxrank} should be increased or the tolerance {\tt tol}
relaxed. The joint selection of
these  two parameters is distinctive of truncation-based strategies.
}
\end{example}

\begin{table}[t]
{\footnotesize
\centering
\begin{tabular}{|c|c|c|r|}
\hline
$\gamma_0$ & {\tt maxrank} & {\tt tol} & $\#$ iter \\
\hline
\multirow{ 2}{*}{$\sin(z\pi)$}
 & 20 & $10^{-6}$ & 5 \\
           & 20 & $10^{-8}$ & 7 \\
\hline
\end{tabular}
\,
\begin{tabular}{|c|c|c|r|}
\hline
$\gamma_0$ & {\tt maxrank} & {\tt tol} & $\#$ iter \\
\hline
\multirow{ 4}{*}{$\exp(z\pi)$}
 & 20 & $10^{-6}$ & 10 \\
           & 20 & $10^{-8}$ &  --  \\
           & 30 & $10^{-8}$ & 17 \\
           & 40 & $10^{-8}$ & 5 \\
\hline
\end{tabular}
\caption{Example~\ref{ex:diffreac}. Number of iterations for \sscg\ as {\tt maxrank} and final tolerance {\tt tol} vary,
for different reactive term coefficient $\gamma_0$.
`` -- '' stands for no convergence in 100 iterations.
\label{tab:pde_analysis}}
}
\end{table}

\begin{example}\label{ex:SSS}
{\rm
We consider an example stemming from the control of dynamical systems,
first discussed in \cite{BennerBreiten2013} and then used in \cite{Shanketal.16} for
comparison purposes.
The matrices correspond to the discretization of a heat model problem
in the spatial domain $(0,1)^2$, so that $\bA$ is the
discretization of the 2D Laplace operator, and $\bM=\bN \bN^\T$ is a low rank 
matrix (with rank the square root of
the dimension of $\bA$), allocating Robin conditions
$\bar{{\bm n}}\cdot \nabla(x) = \delta u(x-1)$ on one of the domain boundaries,
while zero Dirichlet conditions are imposed on the rest of the
boundary. In our experiments we consider 
$\delta \in \{0.5, 0.9\}$.
 The Lyapunov problem for $\delta=0.5$ was called {\sc heat1} in
\cite{Shanketal.16}. We name {\sc heat1}$(\delta)$ the two settings where $\delta=0.5, 0.9$.
We consider two discretization levels leading to
a matrix problem of dimension $n_A=320^2=102 400$ as in~\cite{Shanketal.16}, and
$n_A=500^2=250 000$.
%
Both \sscg\ and \tpcg\ are preconditioned by
$\mathcal{L}_0:\bX\rightarrow \bA\bX + \bX \bA$ by running $t_{ADI}=8$ LR-ADI iterations.
For all considered methods, the accuracy tolerance is $\mathtt{tol}=10^{-6}$.

\begin{table}
\footnotesize{
\begin{center}
\begin{tabular}{|c|c|c|r|r|r|r|}
\hline
Example & $ n_A$ & {\tt maxrank}  & {\sc sss} &  \tpcg\ & \sspcg\ &  \sspcg\  \\
&  &   &(iter/alloc/rank) &   & determ.  & rand'zed \\
\hline
\multirow{ 6}{*}{{\sc heat1}(0.5)}
& $102400$ & $20$ &  & 61.35 ($16$) & -- ($100$) &	-- ($100$) \\ 
&         & $30$ &&  22.78 ($4$) & 17.93 ($3$) & 18.17 ($3$)\\
&         & $$ & {\bf 6.15} ($7/126/31$) &  &  &			 \\
& $250000$ & $20$ &  & -- ($100$) & -- ($100$) &	-- ($100$)  \\ 
&         & $30$ &   & 60.84 ($4$) & 64.46 ($4$) & 64.45 ($4$)\\
&         & $$ & {\bf 18.35} ($7/139/31$) &  &  & 			\\%
\hline
\multirow{ 7}{*}{{\sc heat1}(0.9)}
& $102400$ 	& $40$ & & -- ($100$) & -- ($100$) & -- ($100$) \\
&  	& $50$ & & 310.72 ($26$)   & {\bf 58.11} ($5$) & 58.52 ($5$)\\
&         & -- & -- ($50/\,/\,$) &  &  & \\
& $250000$ & $50$ && 2401.90 ($93$) & -- ($100$) & -- ($100$)   \\
			&& $60$ &  & 936.39 ($30$) &  {\bf 119.55} ($4$) & 120.43 ($4$)\\
&         & -- & -- ($50/\,/\,$) &  &   & \\
\hline
\end{tabular}
\end{center}
\hskip 0.7in -- no conv.
}
\caption{Example~\ref{ex:SSS}. For each method, running time in seconds, and in parenthesis
the number of iterations. Stopping tolerance
{$10^{-6}$}. For {\sc sss}, number of iterations, the subspace total memory allocation 
for length $n_A$ vectors and the solution rank are
 reported.
\label{ex:BB}}

\end{table}

In Table~\ref{ex:BB} we report the results for both 
{\sc heat1}(0.5) and {\sc heat1}(0.9), and different values of 
$n_A$ and {\tt maxrank}. 
For {\sc heat1}(0.5), the (standard) Lyapunov operator $\mathcal{L}_0$ is 
the dominant part of the whole $\mathcal{L}$. This is the scenario {\sc sss} has been 
designed for. Indeed, from the results in Table~\ref{ex:BB} we can see 
that {\sc sss} converges very fast for both values of $n_A$. 

The operator $\mathcal{L}_0$ is less dominant in {\sc heat1}(0.9). Indeed, 
{\sc sss} has convergence problems: after 50 fixed-point iterations its residual
estimate is still above $10^{-3}$.
On the other hand, both \tpcg\ and \sscg\ converge for sufficiently
large values of {\tt maxrank}. In particular, \sscg\ (both determ. and
rand'zed) requires way fewer iterations than \tpcg\ with consequent remarkable 
benefits in terms of computational timings,  being 
about one order of magnitude faster than \tpcg\ for this problem. Due to the 
small number of terms in $\mathcal{L}$ ($\ell=3$), no notable performance difference 
of \sscg\ \emph{determ.} and \emph{rand'zed} is observed
in terms of either computational timings or memory allocations.
}
\end{example}

\subsection{Numerical experiments for the
multiterm Sylvester equation}\label{sec:expes_sylv}

In this section we consider the more general multiterm Sylvester equation
in (\ref{eqn:main}), with one-term or two-term preconditioning (cf. $\mathcal{P}_1$ and $\mathcal{P}_2$ in section~\ref{sec:precond}, respectively).

\begin{example}\label{ex:param}
{\rm 
We consider a problem used in \cite[section 5.1]{Biolietal2024}, consisting 
of the parameterized diffusion equation 
$- \nabla\cdot (k \nabla u) = 0$ in $(0,1)^2$, with homogeneous boundary conditions
and semi-separable diffusion {coefficient 
\vskip -.2in
$$
k(x,y)= \delta_1 k_{1,x}(x) k_{1,y}(y) + \ldots + 
\delta_{\ell_k} k_{\ell_k,x}(x) k_{\ell_k,y}(y), \quad
{\rm where} \quad
k(x,y)=1 + \sum_{j=1}^{\ell_k-1}  \frac{10^j}{j!} x^j y^j,
$$
with $\ell_k=4$.}
We insert $k(x,y)$ into the equation,
and discretize the second order operator using standard finite differences (see
Example~\ref{ex:diffreac}). We then obtain the matrix equation
$$
\sum_{j=1}^{\ell_k} \delta_j ( A_{j,x} \bX \bD_{j,y} + \bD_{j,y} \bX A_{j,y}) = \bC
$$
where $\bC$ is a  rank-four matrix accounting for the boundary conditions;
see \cite{Biolietal2024}. For $\ell_k=4$ a total of $\ell=8$ terms appear
in the matrix equation.
We observe that the operator  is a Lyapunov operator, however the overall
equation is nonsymmetric, due to the nonsymmetry of the right-hand side  matrix $\bC$.
%
In \cite[Section 5.1]{Biolietal2024}, two-term preconditioning of the form $\mathcal{P}_2$ 
(cf. section~\ref{sec:precond}) was employed,
with the specific selection of the third and forth terms in ${\cal L}$.
Following \cite[Section 3.4]{Biolietal2024}, the preconditioning step in  {\sc r-nlcg} 
operates within a metric-related matrix inner product.
For both \sscg\ and \tpcg, $t_{ADI}=8$ LR-ADI iterations were employed for $n_A=10000$
whereas $t_{ADI}=15$ iterations were necessary for $n_A=102400$. The stopping criterion
used $\mathtt{tol}=5\cdot 10^{-6}$ as threshold.

According to the discussion in section~\ref{sec:precond}, for 
a large number $\ell$ of terms
the use of one-term preconditioning  $\mathcal{P}_1$ may be beneficial.
For both \sscg\ and \tpcg, the left third and right forth terms in ${\cal L}$ were
 used to define ${\cal P}_1$.
The results in Table~\ref{ex_parm:table1} show that ${\cal P}_1$ is
extremely effective for both problem dimensions when associated with \sscg, giving at
least one order of magnitude lower timings than with {\sc r-nlcg}.
The systematic
low number of iterations for small {\tt maxrank} is also remarkable. 
{The use of ${\cal P}_1$ is thus recommended.}

\begin{table}
\footnotesize{
\begin{center}
\begin{tabular}{|c|c|c|r|r|r|r|}
	\hline
	$n$ & Precond  & {\tt maxrank}  & {\sc r-nlcg}&  \tpcg\ & \sspcg\ & \sspcg\    \\
	&\rev{type}  &  & &   & determ.  & rand'zed \\
	\hline
	10000 & $\mathcal{P}_1$ & 20 & -- ($100$) & -- ($100$) & -- ($100$) & -- ($100$)  \\ 
	&$\mathcal{P}_1$ & 40 & -- ($100$) & -- ($100$) & 1.08 ($\;\;5$)& {\bf 0.92} ($\;\;5$) \\ 
	&$\mathcal{P}_1$ & 60 & -- ($100$) & -- ($100$) & 2.47 ($\;\;5$)& {\bf 2.34} ($\;\;5$) \\ 
	&$\mathcal{P}_2$  &
	{20} & 	{\bf 11.25 ($36$)} &  	{ 11.42 ($38$)} & -- ($100$) &  -- ($100$) \\
	&$\mathcal{P}_2$  & 40 & {*42.97 ($36$)} & {\bf 15.54} ($33$)& -- ($100$) &  -- ($100$) \\ 
	&$\mathcal{P}_2$ & 60 & {*98.62 ($35$)} & 32.39 ($28$) & 9.59 ($\;\;5$)& {\bf 8.37} ($\;\;5$)  \\ 
	\hline
	102400 & $\mathcal{P}_1$  & 20 &  -- ($100$) & -- ($100$) & -- ($100$) & -- ($100$)  \\ 
	&$\mathcal{P}_1$ & 40 & $\dagger$ & -- ($100$) & 18.17 ($\;\;6$) & {\bf 8.74} ($\;\;6$) \\ 
	&$\mathcal{P}_1$ & 60 & $\dagger$ & -- ($100$) & 23.50 ($\;\;5$) &  {\bf 16.93} ($\;\;5$) \\ 
	&$\mathcal{P}_2$   & 20 & {\bf  183.44} ($41$) &  -- ($100$) & -- ($100$) &  -- ($100$)  \\
	&$\mathcal{P}_2$ & 40 & $\dagger$ &  446.94 ($47$)& -- ($100$) &  -- ($100$) \\ 
	&$\mathcal{P}_2$ & 60 & $\dagger$ & 884.20 ($26$) & 115.73 ($\;\;3$)& {\bf 101.91} ($\;\;3$)  \\ 
	\hline
\end{tabular}

\end{center}
\hskip 0.7in -- no conv. \hskip 0.3in * Lower final residual norm than other methods \hskip 0.3in {$\dagger$ Out of Memory}
}
        \caption{Example~\ref{ex:param}. For each method, running time in seconds, and in parenthesis
the number of iterations. Stopping tolerance
{$\mathtt{tol}=5\cdot10^{-6}$}. 
\label{ex_parm:table1}}
\end{table}
}
\end{example}


\begin{example}\label{ex:stoch}
{\rm 
We consider a multiterm Sylvester equation
arising from the Galerkin approximation of the following two dimensional
 elliptic PDE problem with correlated random inputs,
\begin{eqnarray}\label{eqn:pdestoch}
- \nabla\cdot (a(x,\omega) \nabla u) &=& f\quad \text{in}\,\, D, \quad
u(x,\omega)=0, \quad \text{on}\,\, \partial D.
\end{eqnarray}
Here $\omega\in\Omega$, where $\Omega$ is a sample space associated with a
proper probability space; 
see, e.g., \cite[Chapter 9]{LordPowellShardlow.14}, and $D\subset {\mathbb R}^2$ is the 
space domain.
The diffusion coefficient is assumed to be a random field, with 
expansion 
in terms of a finite number of real-valued independent random variables $\{\xi_j\}_{j\le \ell-1}$
defined in $\Omega$. Following the derivation in \cite{Powelletal2017},
we consider a truncated Karhunen-Lo{\`e}ve expansion, giving
$a(x,\omega)=\mu(x) + \sigma \sum_{j=1}^{\ell-1}  \sqrt{\lambda_j} \phi_j(x) \xi_j(\omega)$,
where $\mu$ corresponds to the diffusion coefficient expected value, $\sigma$ is
the standard deviation, while $(\lambda_j, \phi_j)$ are the leading eigenpairs of
the associated covariance matrix. 
Under proper hypotheses on the coefficients, the
problem is well posed, and its Galerkin finite element discretization on a tensor space
(see \cite{Powelletal2017})
 gives an algebraic problem of type (\ref{eqn:main}) with $\ell$ terms:
 the matrices $\bA_i$ account for the spatial discretization terms,
while the matrices $\bB_i$ contain the discretized weighted moments in the
random basis. The right-hand side is a rank-one matrix $\bC=f_0 e_1^\T$, where
$f_0$ is the finite element discretization of the forcing term in (\ref{eqn:pdestoch}).

In our experiments we first 
consider the data
corresponding
to Example 5.2 and Example 5.1 in \cite{Powelletal2017};
the second example was also used in \cite{Biolietal2024}.  
The $\bA_i$s have size $n_A=16\,129$,
while the $\bB_i$s have size $n_B=1287$ and $n_B=2002$, respectively.
The problems have $\ell=9$ and $\ell=10$ terms, resp.
%
{Explicit inspection (not reported here) shows that the solution 
of Example 5.2 in \cite{Powelletal2017} has about
200 singular values with magnitude above $10^{-6}$,} from which we deduce
that we cannot expect a very accurate approximate solution of small rank.

For this problem we also compare the performance with that
of the algorithm MultiRB from \cite{Powelletal2017}, briefly recalled
at the beginning of section~\ref{Numerical results}.
For this method, the final subspace dimension and the
final solution rank are reported; both are recorded a-posteriori. 
Our experimental results are shown in Table~\ref{ex_stoch:package},
with stopping tolerance $\mathtt{tol}=10^{-6}$ and ${\cal P}_1$ preconditioning with $\bA_1, \bB_1$,
as used in the literature for this problem.
As expected, \sscg\ requires a large value of {\tt maxrank} to converge
smoothly for the first problem. For smaller values, say {\tt maxrank}=100, the
convergence curve reached a plateau right above the requested  tolerance,
so that a slightly larger value of {\tt tol} would have ensured a successful
completion. 
Low memory allocations of the randomized strategy
show that our new approach is also superior to MultiRB, in
addition to {\sc r-nlcg}, while being quite effective in terms of running time,
for both cases. We should mention, however, that for
the first problem, {\tt maxrank}=125 produced a post-computed 
true residual norm larger than that obtained with MultiRB. Using
{\tt maxrank}=150 overcame the problem.
Comparisons with the most direct competitor
\tpcg\ are consistent with the previous results.
We highlight the particularly fast convergence of \sscg\ for the second
dataset, both in terms of number of iterations and running time.

\begin{table}
\footnotesize{
\begin{center}
\begin{tabular}{|c|c|c|r|r|r|r|r|r|}
	\hline
	Example & $n_A,  n_B$ &{\tt maxrank}  & {\sc r-nlcg} & MultiRB&  \tpcg\ & \sspcg\ & \sspcg\ \\
	($\ell=10$)&  &   &  & (spacedim/rank)&   & determ.  &  rand'zed \\
	\hline
	\cite[Ex.5.2]{Powelletal2017}  
	&16129, 1287 & 60 & -- ($100$) &  & -- ($100$) & -- ($100$) &  -- ($100$) \\ 
	&&125 & 67.54 ($38$)&  & 53.50 ($18$) &  19.55 ($12$)  & {\bf 11.83} ($13$) \\ 
	&&150 & 57.31 ($24$) &  &  66.88 ($14$) & 23.73 ($11$)  &  12.58 ($11$) \\ 
	&&  &   &  12.76 ($312/306$) &  &  &     \\ 
	\hline
	\cite[Ex.5.5]{Powelletal2017}
	&16129, 2002 & 25 & {$\star$5.58 ($28$)} &   & { 6.55} ($24$)  & -- ($100$) &  -- ($100$) \\ 
	& & 50 & 14.63 ($26$)  &  & 13.03 ($16$) & 4.89 ($\;\;8$) & {\bf 3.20} ($\;\;8$)\\
	&& 100 & 35.98 ($25$) &  & 37.11 ($16$) & 6.39 ($\;\;6$) & { 3.82} ($\;\;6$) \\
	&&  &  & 6.89 ($158/66$) &  &  & \\
	\hline
\end{tabular}	
\end{center}
\hskip 0.1in  -- no conv.  \hskip 0.3in $\star$ Final residual norm is {\it larger} than for other methods
	}
	\caption{Example~\ref{ex:stoch},  stochastic problem.
For each method, running time in seconds,
and in parenthesis the number of iterations. Stopping tolerance
{$\mathtt{tol}=10^{-6}$}. Best running times are in bold. For MultiRB the
the final approximation space dimension and the final solution rank are reported.
\label{ex_stoch:package}}

\end{table}

We also created  larger datasets for  the setting of
Example 5.1 in \cite{Powelletal2017}, using the {\sc s-ifiss}
package \cite{sifiss}; unless explicitly stated, all default values were
used to create the problem data. We  employed a finer spatial discretization
so that $n_A=65\,025$ (level 8), and used
$9$ random   variables with polynomial  degree $5$, so that $n_B=2002$, as above.
For this  setting  we used both  fast and slow decay  of  the
expansion coefficients; we refer to
\cite{Powelletal2017} for a discussion. 
Finally, we consider the `fast'  and `slow' decay problems with
$9$ random   variables and polynomial  degree $6$, yielding  $n_B=5005$.
All results with these newly created data are reported in Table~\ref{ex_stoch:new}.

In spite of the broader scenario, the results are consistent with
the previous ones, with the new method largely surpassing its more
direct competitors in all instances. The comparison with respect to
MultiRB is less clear cut, if only running time is considered, while overall
the new method requires less memory. We conclude with a comment on the
expected behavior for larger $n_B$ on this problem. Since MultiRB provides
no reduction in the stochastic variable, the costs of the method will 
significantly increase with $n_B$, whereas we expect less dramatic effects
on the new method.

Finally, in Table~\ref{ex_stoch:comparison} we test the performance
of the memory-saving dynamic residual
computation described in section~\ref{sec:truncation} with respect to the full
computation of the residual factor, on this last dataset. 
Although the number of iterations seems to not have been affected, this is not
so for the running time, which in many cases more than doubles.
Similar results were obtained with the previous examples whenever $\ell$ was large.
Summarizing, and
given the especially good behavior of the randomized memory-saving strategy
we have devised, we do not advocate using the dynamic strategy, at least for
the classes  of problems we have tested it.

}
\end{example}

\begin{table}
\footnotesize{
\begin{center}
\begin{tabular}{|c|c|c|r|r|r|r|r|r|}
	\hline
	$n_A, n_B$& decay & {\tt maxrank}  & {\sc r-nlcg}&  MultiRB & \tpcg\ & \sspcg\ &  \sspcg\  \\
	&  &   & & (spacedim/rank)&   & determ.  & rand'zed \\
	\hline
	65025, 2002&fast &20 & {$\star$50.39 ($29$)} &  & -- ($100$) &-- ($100$) &-- ($100$)\\ 
	&      &30 & 82.81 ($27$) &  & 27.65 ($16$)  & 10.61 ($10$) &  {\bf 8.29} ($11$) \\ 
	&      &40 & 111.98 ($27$) & & 38.02 ($16$)  & 13.54 ($\;\;9$) &  {9.30} ($\;\;9$)  \\ 
	&      & &  &{21.02} (${159}/66$) &   &  &  \\ 
	\hline
	65025, 2002&slow &40 &  {$\star$197.97 ($45$)} &  & 59.22 ($24$) &  -- ($100$) &  -- ($100$)\\ 
	&      &50 &  194.77 ($33$) &  & 41.26 ($12$) & 21.25 ($\;\;9$) &  { 13.05} ($\;\;9$) \\ 
	&      &60 & 187.93 ($24$)  & & 48.27 ($11$) & 22.89 ($\;\;8$)  &  15.34 ($\;\;8$) \\ 
	&      & &   & {\bf 10.04} (${124}/101$)&  &   &  \\ 
	\hline
	65025, 5005&fast  &20 & {$\star$29.21 ($37$)} &   &-- ($100$) &-- ($100$) & -- ($100$)\\ 
	&  &30 & 33.06 ($25$) &  & 38.09 ($19$) & 14.79 ($13$) &  { 11.02} ($14$) \\ 
	&  &40 & 44.27 ($25$) &  &  42.38 ($17$) & 15.70 ($10$) &  {\bf 10.66} ($10$)  \\ 
	&  & & &29.79 ($159/72$)  & &  &    \\ 
	\hline
	65025, 5005&slow &40 & {$\star$39.17 ($22$)}  &  & 162.436 ($56$) &  -- ($100$)  & -- ($100$) \\ 
	&      &50 & 45.52 ($19$) &  & 50.15 ($13$) & 24.41 ($10$) &  { 14.79} ($10$) \\ 
	&      &60 &  59.57 ($18$) & & 55.02 ($12$) & 33.15 ($\;\;9$) &  21.39 ($\;\;9$) \\ 
	&      & & &{\bf 13.32 ($133/107$)}  & &  &   \\ 
	\hline
\end{tabular}
\end{center}
\hskip 0.2in  -- no conv. \hskip 0.3in $\star$ Final residual norm is {\it larger} than for other methods
}
\caption{Example \ref{ex:stoch},  stochastic problem, large dimensions. For each method,
running time in seconds, and in parenthesis the number of iterations.
 For MultiRB, in parenthesis are approximation space dimension and final solution rank.
Stopping tolerance
{$\mathtt{tol}=10^{-6}$}. Best running times are in bold. \label{ex_stoch:new} }
\end{table}

\begin{table}
{
\footnotesize{
\begin{center}
\begin{tabular}{|c|c|c|r|r||c|c|c|r|r|}
	\hline
	$n_A,$ & decay & {\tt maxrank}   & \sspcg\ & \sspcg\    &
	$n_A,$ & decay & {\tt maxrank}   & \sspcg\ & \sspcg\   \\
	$n_B$&  &     & determ.  & dyn.  &
	$n_B$&  &     & determ.  & dyn.  \\
	\hline
	65025, & fast       &30 &  10.61 ($10$) & 30.76 ($10$) & 
	65025, & fast   &30 &   14.79 ($13$) & 47.33 ($13$) \\
	2002&       &40 &  13.54 ($\;\;9$) &  36.31  ($\;\;9$)   &
	5005&   &40 &   15.70 ($10$) &  47.63 ($10$)  \\
	\hline
	& slow      &50 &  21.25 ($\;\;9$) & 61.41 ($\;\;9$) & 
	& slow      &50 &   24.41 ($10$) &  66.93 ($10$) \\
	&       &60 &  22.89 ($\;\;8$) & 65.44 ($\;\;8$) &
	&       &60 &   33.15 ($\;\;9$) &  70.77 ($\;\;9$) \\
	\hline
\end{tabular}
\end{center}
}
\caption{Example \ref{ex:stoch},  stochastic problem, large dimensions. Comparison
between storing the whole residual factor and dynamically  updating its truncated
QR factorization. For each method,
running time in seconds, and in parenthesis the number of iterations.
Stopping tolerance {$\mathtt{tol}=10^{-6}$}.\label{ex_stoch:comparison} }
}
\end{table}

\section{Conclusions}\label{Conclusions}
We have proposed a new iterative method for solving multiterm matrix equations
with symmetric and positive definite operator. The method generates
a sequence of approximate solutions by locally minimizing a functional
over a subspace that is allowed to grow up to a desired threshold.
The derivation closely follows that of matrix-oriented Conjugate Gradients on the
Kronecker form, without being affected by the same dramatic loss of optimality.
By using particularly convenient randomized range-finding strategies, 
the method is able to ensure low memory requirements.
Our numerical experiments have shown that the new method is computationally
robust and competitive with respect to state-of-the-art methods in addition to {\sc tcg},
on quite diverse application problems.

The matlab code of \sscg\ \rev{is} available
at 
{\tt https://github.com/palittaUniBO}.

\section*{Acknowledgments}
\rev{The authors thank the reviewers for their careful 
reading of the manuscript and their insightful comments.}
All authors are members of the INdAM Research
Group GNCS. Moreover, their work
was partially supported by the European Union - NextGenerationEU under the National Recovery and Resilience Plan (PNRR) - Mission 4 Education and research
- Component 2 From research to business - Investment 1.1 Notice Prin 2022 - DD N. 104 of 2/2/2022,
entitled “Low-rank Structures and Numerical Methods in Matrix and Tensor Computations and their
Application”, code 20227PCCKZ – CUP J53D23003620006.

\bibliographystyle{siamplain}
\bibliography{references}
\section*{Appendix}

In this Appendix we 
report in Algorithm~\ref{alg:Tres} the low-rank truncation 
scheme for the residual matrix presented in section~\ref{sec:truncation}.

\begin{algorithm}[hbt]
{\footnotesize
\begin{algorithmic}[1]
\smallskip
\Statex \textbf{Input:} Operator
$\mathcal L:\mathbb{R}^{n_A\times n_B}\rightarrow \mathbb{R}^{n_A\times n_B}$,
factors $C_1$, $C_2$ $(\bC=C_1C_2^{\T})$,
factors of current approx sol'n $X^l$, $X^r$, $\bm{\tau}$
$(\bX=X^l\bm{\tau}(X^r)^{\T})$, matrices $\Omega$ and $\Pi$,
 truncation tolerances {\tt tolrank, maxrank}.
\Statex \textbf{Output:} Approximate
factors  $R^l$, $R^r$, $\bm{\rho}$ to the residual matrix $\bC-\mathcal{L}(\bX)$
\smallskip

\State (Short-hand not'n:
{\footnotesize$A_\star\bullet X^l = [A_1X^l,\ldots,A_\ell X^l]$,
$B_\star\bullet X^r = [B_1X^r,\ldots,B_\ell X^r]$,
${\bm D}=\text{blkdiag}(I_{s_C},-\bm{\tau},\ldots,-\bm{\tau})$})
\If{$\Omega=\Pi=\emptyset$}
\State $[R^l,\bm{\rho},R^r]=
\mathcal{T}([C_1, A_\star\bullet X^l], {\bm D},
[C_2,B_\star \bullet X^r], \mathtt{tolrank}, \mathtt{maxrank})$
\Else
\State Compute skinny QRs$^{(\dagger)}$
{\footnotesize
$$
[Q,*]=\mbox{\sc{qr}}([C_1,A_\star\bullet X^l]
\left({\bm D}
([C_1,B_\star \bullet X^r]^{\T}\Omega)\right)), \quad
 [G, *]=\mbox{{\sc qr}}([C_2,B_\star\bullet X^r]
\left({\bm D}
([C_1,A_\star\bullet X^l]^{\T}\Pi)\right))
$$
}
\State Compute truncated SVD based on {\tt tolrank} and {\tt maxrank}:
{\scriptsize 
$
\quad U\Sigma V\approx Q^{\T}[C_1,A_\star\bullet X^l] {\bm D}
[C_2,B_\star \bullet X^r]^{\T}G
$
}
\State Set $R^l=QU$, $\bm{\rho}=\Sigma$, $R^r=GV$
\EndIf
8
\end{algorithmic}
}
\caption{$\mathcal{T}_{res}$ - truncation procedure for the residual matrix}\label{alg:Tres}
${}^{(\dagger)}$
\footnotesize{The product
$[C_2,B_\star \bullet X^r]^{\T}\Omega$ is performed one block at the
time, so as not to explicitly form
$B_\star \bullet X^r$. The same for all other similar products in the
algorithm.}

\end{algorithm}

\end{document}